\documentclass[a4paper]{article}
\usepackage[all]{xy}
\usepackage[utf8]{inputenc}        
\usepackage[dvips]{graphics,graphicx}
\usepackage{amsfonts,amssymb,amsmath,color,mathrsfs, amstext}
\usepackage{amsbsy, amsopn, amscd, amsxtra, amsthm,authblk}
\usepackage{enumerate,algorithmicx,algorithm}
\usepackage{algpseudocode}
\usepackage{upref}
\usepackage{natbib}
\usepackage{geometry}
\usepackage[displaymath]{lineno}
\usepackage{float}
\usepackage{yhmath}
\usepackage{booktabs}
\usepackage{subcaption}
\usepackage{multirow}
\usepackage{makecell}
\usepackage[colorlinks,
            linkcolor=blue,
            anchorcolor=blue,
            citecolor=blue
            ]{hyperref}

\numberwithin{equation}{section}

\newcommand\keywords[1]{\textbf{Keywords:} #1}
\newcommand{\MSC}[1]{\textbf{MSC2020:} #1}

\newtheorem{theorem}{Theorem}[section]
\newtheorem*{theorem*}{Theorem}
\newtheorem{lemma}{Lemma}[section]

\newtheorem{remark}{Remark}[section]

\numberwithin{equation}{section}

\makeatletter
\newcommand{\Rmnum}[1]{\expandafter\@slowromancap\romannumeral #1@}
\makeatother
\makeatother

\begin{document}

\title{An Efficient Space-Time Two-Grid Compact Difference Scheme for the Two-Dimensional Viscous Burgers' Equation}

\author[1]{Xiangyi Peng\thanks{E-mail: pxymath18@smail.xtu.edu.cn}}
\author[1]{Lisen Ding\thanks{E-mail: dingmath15@smail.xtu.edu.cn}}
\author[2]{Wenlin Qiu\thanks{E-mail: qwllkx12379@163.com; Corresponding author.}}
\affil[1]{Hunan Key Laboratory for Computation and Simulation in Science and Engineering, School of Mathematics and Computational Science, Xiangtan University, Xiangtan 411105, China.}
\affil[2]{School of Mathematics, Shandong University, Jinan 250100, China.}

\maketitle

\begin{abstract}
  This work proposes an efficient space-time two-grid compact difference (ST-TGCD) scheme for solving the two-dimensional (2D) viscous Burgers' equation subject to initial and periodic boundary conditions. The proposed approach combines a compact finite difference discretization with a two-grid strategy to achieve high computational efficiency without sacrificing accuracy. In the coarse-grid stage, a fixed-point iteration is employed to handle the nonlinear system, while in the fine-grid stage, linear temporal and cubic spatial Lagrange interpolations are used to construct initial approximations. The final fine-grid solution is refined through a carefully designed linearized correction scheme. Rigorous analysis establishes unconditional convergence of the method, demonstrating second-order accuracy in time and fourth-order accuracy in space. Numerical experiments verify the theoretical results and show that the ST-TGCD scheme reduces CPU time by more than 70\% compared with the traditional nonlinear compact difference (NCD) method, while maintaining comparable accuracy. These findings confirm the proposed scheme as a highly efficient alternative to conventional nonlinear approaches.
\end{abstract}

\keywords{Space-time two-grid method, compact difference scheme, 2D Burgers' equation, convergence analysis, high-order accuracy.} 

\vskip 2mm
\MSC{65M06, 65M15, 65M55}

\section{Introduction}
In this paper, we consider the numerical solution of the 2D viscous Burgers' equation (BE) by using a space-time two-grid compact difference method. The governing equation is given by
\begin{equation}\label{main_equation}
u_t+u(u_x+u_y)-\lambda{\Delta} u=0, \quad (x,y)\in \mathbb{R}^2, \quad 0<t\leq T,
\end{equation}
subject to the initial condition
\begin{equation}\label{initial_condition}
u(x,y,0)=u_0(x,y), \quad (x,y)\in \mathbb{R}^2,
\end{equation}
and the periodic boundary conditions
\begin{equation}\label{boundary_condition}
u(x,y,t)=u(x+L_1,y,t), \quad u(x,y,t)=u(x,y+L_2,t), \quad (x,y)\in\mathbb{R}^2, \quad 0<t\leq T,
\end{equation}
 where ${\Delta}$ is the classic spatial Laplacian operator, $\lambda>0$ is the kinematic viscosity coefficient, $L_1$ and $L_2$ denote the positive periods in the $x$- and $y$-directions, respectively.

The numerical simulation of nonlinear convection-diffusion partial differential equations (PDEs) has long been a central topic in computational fluid dynamics. The Burgers' equation has received extensive attention within this field. First formulated by Bateman \cite{bateman1915some} in 1915, the equation was later employed in 1948 by the Dutch physicist Burgers as a mathematical model for turbulence \cite{burgers1948mathematical}, as represented by \eqref{main_equation}. In recognition of Burgers’s contribution, the equation now bears his name. Subsequent studys have revealed that the BE serves as a simplified  model for a wide range of physical phenomena, including shock waves \cite{kreiss1986convergence,Laforgue}, gas dynamics \cite{Brio,Kundu}, traffic flow \cite{musha1978traffic,yu2002analysis}, and so on.

Although the BE can be converted into the linear heat equation via the Hopf-Cole transformation and an exact solution can therefore be derived \cite{Fletcher}, this analytical solution is expressed as an infinite series that is too complex for practical use. Consequently, numerical methods have been developed to approximate the BE, which not only provide an alternative strategy for solving the BE itself, but also serve as a blueprint for tackling more complex equations such as Navier–Stokes equations \cite{LBYtwo}, Korteweg–de Vries (KdV) equation \cite{miles1981korteweg} and Kuramoto–Sivashinsky (KS) equation \cite{akrivis1992finite}. Hence, devising highly accurate and efficient numerical approaches for the BE are of fundamental significance.

Over the past few decades, a wide variety of numerical techniques have been proposed, including finite-element methods \cite{Varoglu,caldwell1987solution}, B-spline finite-element methods \cite{ALI1992325,kutluay2004numerical}, finite-difference methods \cite{bahadir2003fully,xu2009second}, spectral methods \cite{Guo1,Guo2}, spline collocation methods \cite{daug2005numerical,arora2013numerical}, operator splitting methods \cite{holden1999operator,holden2013operator} and so forth. Many of these approaches lead to nonlinear discrete systems that must be solved by Newton or Picard iterations. When the spatial mesh is fine, these nonlinear solvers require repeated assembly and solution of large algebraic systems, resulting in expensive computational costs. Several researchers choose to directly construct linear schemes, or reformulate the original nonlinear schemes into linear ones. For instance, Sun et al. \cite{sun2015two} constructed and analyzed two linear difference schemes for the BE; Wang et al. \cite{wangxp} proposed both a nonlinear compact difference scheme and its linearized counterpart; and Zhang et al. \cite{ZQF} devised a linearized compact difference method for the two-dimensional Sobolev equation with Burgers convection terms. While these linearized schemes significantly accelerate the computation compared with their nonlinear counterparts, they generally sacrifice accuracy. This trade-off vividly illustrates a perennial dilemma in scientific computing: balancing computational precision against efficiency remains an enduring challenge for researchers.

Two-grid method has emerged as a powerful tool for solving nonlinear PDEs. The core idea is to first solve the nonlinear problem on a coarse grid, and then use this solution to linearize the problem on a finer grid. This approach effectively reduces the computational burden associated with solving large nonlinear systems, while maintaining high accuracy. The two-grid method was first introduced by Xu \cite{Xu1} in 1994 for elliptic problems, and has since been extended to various types of equations, including parabolic and hyperbolic problems \cite{Xu2}. In recent years, two-grid method has been successfully applied to the PDEs with Burgers' nonlinear term $uu_x$. For example, Hu et. al. \cite{HuX} developed a spatial two-grid with mixed finite-element method for the BE. Chen et. al. \cite{chen2023two} applied a temporal two-grid finite difference method to solve one-dimensional (1D) fourth-order Sobolev-type equation with Burgers’ type nonlinearity. Peng et. al. \cite{peng2024novel} proposed a temporal two-grid compact difference method for the 1D Burgers' equation. However, as the above works, the two-grid technique was employed exclusively in either the temporal or the spatial direction, not in both at the same time. To the best of our knowledge, there are few works on space-time two-grid methods for solving the PDEs with the Burgers' nonlinear term. Shi et. al. \cite{shi2024construction} proposed a new space-time two-grid method for the 1D generalized Burgers’ equation. Gao et. al. \cite{gao2025efficient} developed  a space–time two-grid difference scheme for solving the symmetric regularized long wave equation. Nevertheless, both studies are restricted to standard (non-compact) difference schemes and address only the 1D case.

In this work, we aim to develop a novel space-time two-grid compact difference method for solving the 2D BE. the three principal steps are outlined below (see Section \ref{sec3} for full details).
\begin{itemize}
\renewcommand{\labelitemi}{} 
  \item \textbf{i. Coarse-grid solution:} A nonlinear compact difference scheme is solved on a coarse grid by using fixed-point iteration. 
  \item \textbf{ii. Interpolation:} The coarse solution is interpolated onto the fine grid using linear interpolation in time and cubic interpolation in space. 
  \item \textbf{iii. Fine-grid correction:} A linearized compact scheme is solved on the fine grid to obtain a high-accuracy numerical solution. 
\end{itemize}

Figure \ref{fig0} intuitively illustrates the two-grid algorithm step by step, using the spatial direction as an example, the temporal direction is analogous. First, yielding the coarse-grid solution at the black mesh nodes on the left subgraph; then, obtaining the interpolation solution at the red mesh nodes on the middle subgraph; finally, acquiring the corrected solution at the green mesh nodes on the right subgraph. 

 \begin{figure}[h]
     \centering
  \includegraphics[width=\textwidth]{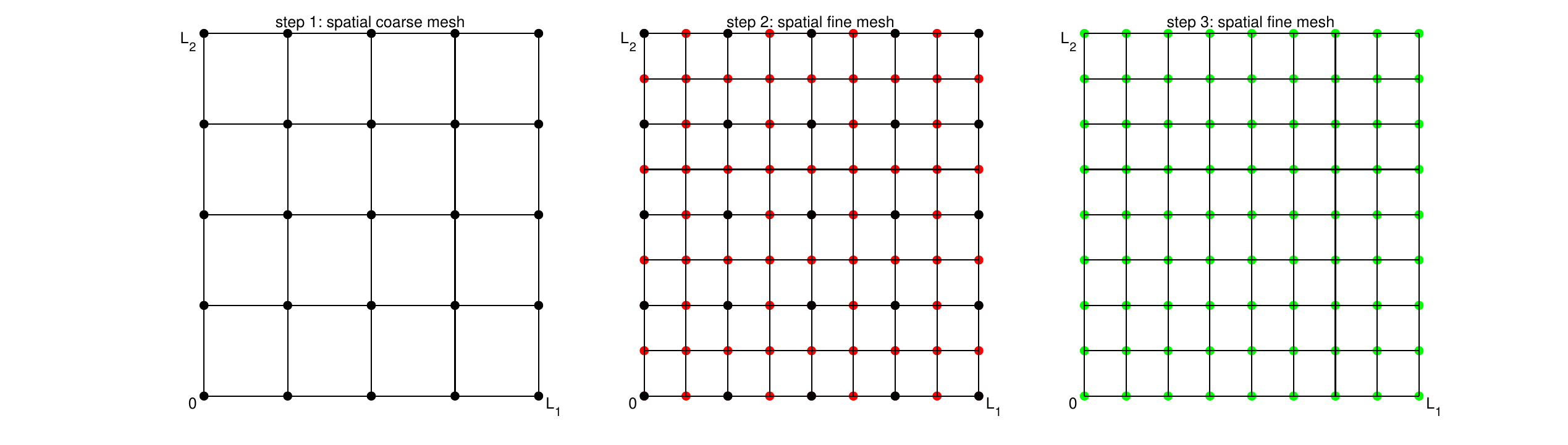}
  \caption{Spatial two-grid computing process, where we suppose $L_1=L_2,~M_1^c=M_2^c=4$, and the spatial step-size ratio $k_h=2$.}\label{fig0}
 \end{figure}

The main contributions of this work are summarized as follows.
\begin{itemize}
    \item  The numerical analysis and computation of the compact difference scheme for the 1D Burgers' equation, previously established in the literature \cite{wangxp}, are extended herein to the 2D case. 

    \item  To the best of our knowledge, this work presents the first fully space-time two-grid compact difference method for 2D nonlinear convection-diffusion problems. By integrating a space-time two-grid strategy with a fourth-order compact discretization for the 2D Burgers' equation, the proposed algorithm significantly reduces computational cost while maintaining the accuracy of the conventional nonlinear compact scheme.

    \item  A rigorous theoretical analysis of the proposed ST-TGCD scheme is conducted. Under reasonable assumptions, the unique solvability of the numerical solution is established. Furthermore, the scheme is proven to be unconditionally convergent, achieving second-order accuracy in time and fourth-order accuracy in space, which is also confirmed by numerical experiments. And, the stability of the fine-grid stage of the ST-TGCD scheme is analyzed.
\end{itemize}

The rest of this paper is organized as follows. In Section \ref{sec2}, we introduce some notations and lemmas that will be used later. In Section \ref{sec3}, we present the construction of the ST-TGCD scheme for the 2D BE. In Section \ref{sec4}, we analyze the uniquely solvability, convergence and stability of the ST-TGCD scheme. In Section \ref{sec5}, we present several numerical experiments to demonstrate the effectiveness of the ST-TGCD method. Finally, we give a short summary in Section \ref{sec6}.

\section{Notations and lemmas} \label{sec2}
To design a space-time two-grid algorithm, we need two sets of grids both in temporal and spatial directions. On the coarse grid, let $h_{c}$ be the coarse-spatial step size, and assume that there exist $M_1^c,~M_2^c\in\mathbb{N}^+$ to satisfy $L_1=h_c M_1^c,~L_2=h_c M_2^c$. Note that this assumption is easily implemented in practice and greatly facilitates both the construction of the scheme and the subsequent error analysis. Denote the coarse-temporal step size $\tau_c={T}/{N^c}$ for a positive integer $N^c$. The coarse grid points are defined as $x_p=ph_{c},~0\leq p \leq M_1^c-1$; $y_q=qh_{c},~0\leq q \leq M_2^c-1$; $t_r=r\tau_c,~0\leq r \leq N^c$. The spatial coarse mesh is denoted by $\Omega_h^c=\{(x_p,y_q)~|~1\leq p\leq M_1^c,1\leq q \leq M_2^c\}$, and temporal coarse grid $\Omega_\tau^c=\{t_r~|~0\leq r \leq N^c\}$. 

On the fine grid,  let $h_{f}=h_c/k_h,~\tau_f=\tau_c/k_\tau$ be the fine spatial and temporal step sizes, where $k_h$ and $k_\tau$ are called as the spatial and temporal step-size ratios, respectively. The numbers of space-time division on the fine grid are $M_1^f=k_h M_1^c,~ M_2^f=k_h M_2^c, ~N^f=k_\tau N^c$. The fine grid points are defined as $x_i=ih_{f},~0\leq i \leq M_1^f-1$; $y_j=jh_{f},~0\leq j \leq M_2^f-1$; $t_n=n\tau_f,~0\leq n \leq N^f$. The spatial fine mesh is denoted by $\Omega_h^f=\{(x_i,y_j)~|~1\leq i\leq M_1^f, 1\leq j \leq M_2^f\}$, and temporal fine grid $\Omega_\tau^f=\{t_n~|~0\leq n \leq N^f\}$. Obviously, the set of coarse-grid nodes is a subset of the fine-grid nodes, i.e., $\Omega_h^c\subseteq \Omega_h^f$ and $\Omega_\tau^c\subseteq \Omega_\tau^f$.  

Define the coarse grid function as $v_c=\{(v_c)_{pq}^r~|~(x_p,y_q)\in\Omega_h^c,t_r\in \Omega_\tau^c \}$, similarly, the fine grid function as $v_f=\{(v_f)_{ij}^n~|~(x_i,y_j)\in\Omega_h^f,t_n\in \Omega_\tau^f \}$. Introduce the following notations for time derivatives:
\begin{equation} \label{time_derivative}
   \begin{split}
     &(v_c)_{pq}^{r-\frac{1}{2}}=\frac{1}{2}\left[(v_c)_{pq}^{r}+(v_c)_{pq}^{r-1}\right],\quad \delta_t^c (v_c)_{pq}^{r-\frac{1}{2}}=\frac{1}{\tau_c}\left[(v_c)_{pq}^{r}-(v_c)_{pq}^{r-1}\right], \\
&(v_f)_{ij}^{n-\frac{1}{2}}=\frac{1}{2}\left[(v_f)_{ij}^{n}+(v_f)_{ij}^{n-1}\right],\quad \delta_t^f (v_f)_{ij}^{n-\frac{1}{2}}=\frac{1}{\tau_f}\left[(v_f)_{ij}^{n}-(v_f)_{ij}^{n-1}\right].  
   \end{split}
\end{equation}

To simplify the notations, we set $v_{ab}^d=(v_c)_{pq}^{r}$ or $(v_f)_{ij}^{n}$, $(x_a,y_b,t_d)\in \Omega_h^\sigma \times \Omega_\tau^\sigma$, where $\sigma=c, f$. Then, the notations for time derivatives in \eqref{time_derivative} can be rewritten as 
\begin{align*}
&v_{ab}^{d-\frac{1}{2}}=\frac{1}{2}\left(v_{ab}^{d}+v_{ab}^{d-1}\right),\quad \delta_t^\sigma v_{ab}^{d-\frac{1}{2}}=\frac{1}{\tau_\sigma}\left(v_{ab}^{d}-v_{ab}^{d-1}\right),\quad \sigma=c,f.
\end{align*}
The notations for the space derivatives are defined as follows:
\begin{align*}
&\delta_x^\sigma v_{a-\frac{1}{2},b}^{d}=\frac{1}{h_{\sigma}}\left(v_{ab}^{d}-v_{a-1,b}^{d}\right), \quad \delta_y^\sigma v_{a,b-\frac{1}{2}}^{d}=\frac{1}{h_{\sigma}}\left(v_{ab}^{d}-v_{a,b-1}^{d}\right),\\
&\widehat{\delta}_x^\sigma v_{ab}^{d}=\frac{1}{2h_{\sigma}}\left(v_{a+1,b}^{d}-v_{a-1,b}^{d}\right), \quad \widehat{\delta}_y^\sigma v_{ab}^{d}=\frac{1}{2h_{\sigma}}\left(v_{a,b+1}^{d}-v_{a,b-1}^{d}\right),\\
&\delta_{xx}^\sigma v_{ab}^{d}=\frac{1}{h_{\sigma}}\left(\delta_x^\sigma v_{a+\frac{1}{2},b}^{d}-\delta_x^\sigma v_{a-\frac{1}{2},b}^{d}\right), \quad \delta_{yy}^\sigma v_{ab}^{d}=\frac{1}{h_{\sigma}}\left(\delta_y^\sigma v_{a,b+\frac{1}{2}}^{d}-\delta_y^\sigma v_{a,b-\frac{1}{2}}^{d}\right),\\
&\widehat{\delta}_h^\sigma v_{ab}^{d}=(\widehat{\delta}_x^\sigma+\widehat{\delta}_y^\sigma)v_{ab}^{d},\quad \sigma=c,f.
\end{align*}
Furthermore, denote $(u v)_{ab}=u_{ab}v_{ab}$ and introduce useful bilinear operators as
\begin{equation} \label{bilinear_operators}
\begin{aligned}
&\psi^\sigma_x(u,v)_{ab} = \frac{1}{3}\left(u_{ab}\widehat{\delta}_x^\sigma v_{ab} + \widehat{\delta}_x^\sigma(u v)_{ab}\right), \quad 1\leq a\leq M_1^\sigma,~1\leq b\leq M_2^\sigma,\\
&\psi^\sigma_y(u,v)_{ab} = \frac{1}{3}\left(u_{ab}\widehat{\delta}_y^\sigma v_{ab} + \widehat{\delta}_y^\sigma(u v)_{ab}\right), \quad 1\leq a\leq M_1^\sigma,~1\leq b\leq M_2^\sigma,\\
&\psi_h^\sigma(u,v)_{ab} = \frac{1}{3}\left(u_{ab}\widehat{\delta}^\sigma_h v_{ab} + \widehat{\delta}^\sigma_h(u v)_{ab}\right), \quad 1\leq a\leq M^\sigma_1,~1\leq b\leq M^\sigma_2,
\end{aligned}
\end{equation}
where $\sigma=c$ or $f$. The bilinear operators $\psi^\sigma_x,\psi^\sigma_y,\psi^\sigma_h$ are used to approximate the nonlinear term $u(u_x+u_y)$ in \eqref{main_equation}. Denote the mesh-function spaces
\begin{align*}
\mathcal{V}_h^\sigma=\{v~|~v=\{v_{ab}\},v_{a+M_1^\sigma,b}=v_{ab},~v_{a,b+M_2^\sigma}=v_{ab},~(x_a,y_b)\in\Omega_h^\sigma, \sigma=c,f\},
\end{align*}
where $\mathcal{V}_h^c$ and $\mathcal{V}_h^f$ denotes the coarse and fine mesh-function spaces, respectively.   
For any mesh functions $u,v\in {\mathcal{V}}_h^\sigma$, we define the following inner products
\begin{equation*}
\begin{split}
   &\langle u,v \rangle_\sigma=h_{\sigma}^2 \sum_{a=1}^{M_1^\sigma}\sum_{b=1}^{M_2^\sigma} u_{ab}v_{ab}, \quad (\delta_x^\sigma u,\delta_x^\sigma v)_\sigma=h_{\sigma}^2 \sum_{a=1}^{M_1^\sigma}\sum_{b=1}^{M_2^\sigma} (\delta_x^\sigma u_{a-\frac{1}{2},b})(\delta_x^\sigma v_{a-\frac{1}{2},b}),\\
   &(\delta_y^\sigma u,\delta_y^\sigma v)_\sigma=h_{\sigma}^2 \sum_{a=1}^{M_1^\sigma}\sum_{b=1}^{M_2^\sigma} (\delta_y^\sigma u_{a,b-\frac{1}{2}})(\delta_y^\sigma v_{a,b-\frac{1}{2}}), \quad \sigma=c,f,
\end{split}
\end{equation*}
and the corresponding norms (seminorm) as
\begin{equation*}
   \begin{split}
   &\|v\|_\sigma=\sqrt{\langle v,v \rangle_\sigma}, \qquad \qquad \|\delta_x^\sigma v\|_{\sigma}=\sqrt{(\delta_x^\sigma v,\delta_x^\sigma v)_\sigma}, \\
   & \|\delta_y^\sigma v\|_{\sigma}=\sqrt{(\delta_y^\sigma v,\delta_y^\sigma v)_\sigma},\quad |v|_{1,\sigma}=\sqrt{\|\delta_x^\sigma v\|_{\sigma}^2+\|\delta_y^\sigma v\|_{\sigma}^2}.
   \end{split}
\end{equation*}
In the following, we list some helpful lemmas to derive the error estimates for the ST-TGCD scheme.
\begin{lemma} \label{lem2.1}
  \cite{Sunbook} For any mesh functions $u,v\in {\mathcal{V}}_h^\sigma~(\sigma=c,f)$, the following identities hold
\begin{equation*}
\begin{split}
    &\langle u, \delta_{xx}^\sigma v \rangle_\sigma=-(\delta_x^\sigma u, \delta_x^\sigma v)_\sigma,\qquad \langle u, \delta_{yy}^\sigma v \rangle_\sigma=-(\delta_y^\sigma u, \delta_y^\sigma v)_\sigma, \\
    &\langle u, \widehat{\delta}_{x}^\sigma v \rangle_\sigma= -\langle\widehat{\delta}_{x}^\sigma u,  v \rangle_\sigma, \qquad \langle u, \widehat{\delta}_{y}^\sigma v \rangle_\sigma= -\langle\widehat{\delta}_{y}^\sigma u,  v \rangle_\sigma.
\end{split}
\end{equation*}
\end{lemma}


\begin{lemma} \label{lem2.2}
    For any mesh function $v\in {\mathcal{V}}_h^\sigma,~(\sigma=c,f)$, the following estimates hold
   \begin{equation} \label{equ2.3}
\|\widehat{\delta}_x^\sigma v\|_{\sigma}\leq \|{\delta}_x^\sigma v\|_{\sigma}, \quad \|\widehat{\delta}_y^\sigma v\|_{\sigma}\leq \|{\delta}_y^\sigma v\|_{\sigma}, \quad  \|\widehat{\delta}_h^\sigma v\|_{\sigma}\leq  \sqrt{2} |v|_{1,\sigma}.
   \end{equation}
\end{lemma}
\begin{proof}
   According to the definitions of $\widehat{\delta}_x^\sigma v_{ab}$ and ${\delta}_x^\sigma v_{ab}$, we have
\begin{equation*}
\begin{split}
\widehat{\delta}_x^\sigma v_{ab}=\frac{1}{2h_{\sigma}}\left[v_{a+1,b}-v_{a-1,b}\right]=\frac{1}{2}\left(\delta_x^\sigma v_{a-\frac{1}{2},b}+\delta_x^\sigma v_{a+\frac{1}{2},b}\right),
\end{split}
\end{equation*}
which implies that
\begin{equation*}
\begin{split}
\|\widehat{\delta}_x^\sigma v\|_{\sigma} \leq \frac{1}{2}\|\delta_x^\sigma v\|_{\sigma}+\frac{1}{2}\|\delta_x^\sigma v\|_{\sigma}=\|\delta_x^\sigma v\|_{\sigma}.
\end{split}
\end{equation*}
Thus, we get the first inequality in \eqref{equ2.3}. The second inequality can be proved similarly. For the third inequality, we have
\begin{equation*}
\begin{split}
\|\widehat{\delta}_h^\sigma v\|_{\sigma} &= \|\widehat{\delta}_x^\sigma v+\widehat{\delta}_y^\sigma v\|_{\sigma} \leq \|\widehat{\delta}_x^\sigma v\|_{\sigma}+\|\widehat{\delta}_y^\sigma v\|_{\sigma} \leq \|\delta_x^\sigma v\|_{\sigma}+\|\delta_y^\sigma v\|_{\sigma} \leq \sqrt{2} |v|_{1,\sigma}.
\end{split}
\end{equation*}
Thus, we have proved the lemma.
\end{proof}

\begin{lemma} \label{lemma2.3}
   \cite{Sunbook} For any mesh function $v\in {\mathcal{V}}_h^\sigma,~(\sigma=c,f)$, we have the inverse estimates
\begin{equation*}
\|\delta_x^\sigma v\|_\sigma \leq \frac{2}{h_{\sigma}}\|v\|_{\sigma}, \quad \|\delta_y^\sigma v\|_\sigma \leq \frac{2}{h_{\sigma}}\|v\|_{\sigma}.
\end{equation*}   
And, together with Lemma \ref{lem2.2}, we have
\begin{equation} \label{eqt2.6}
\|\widehat{\delta}_x^\sigma v\|_\sigma \leq \frac{2}{h_{\sigma}}\|v\|_{\sigma},\quad  \|\widehat{\delta}_y^\sigma v\|_\sigma \leq \frac{2}{h_{\sigma}}\|v\|_{\sigma}.
\end{equation}
\end{lemma}
\begin{lemma} \label{lemma2.4}
   \cite{ZQF} For any mesh functions $u,v\in {\mathcal{V}}_h^\sigma,~(\sigma=c,f)$, we have
   \begin{equation*}
  \langle\psi_x^\sigma(u,v), v \rangle_\sigma = 0, \quad \langle\psi_y^\sigma(u,v), v \rangle_\sigma = 0,  \quad \langle\psi_h^\sigma(u,v), v \rangle_\sigma = 0.
   \end{equation*}
\end{lemma}

\begin{lemma} \label{lemma2.5}
   Let $g(x,y)\in C^5\big([x_{a-1},x_{a+1}]\times[y_{b-1},x_{b+1}]\big)$, $G(x,y):= g_{xx}(x,y)$ and $H(x,y):=g_{yy}(x,y)$, then we have
\begin{equation*}
   \begin{split}
      &g(x_a,y_b)\Big(g_x(x_a,y_b)+g_y(x_a,y_b)\Big)\\
       =~&\psi_h^\sigma(g_{ab},g_{ab})-\frac{h_{\sigma}^2}{2}\left[\psi_x^\sigma(G_{ab},g_{ab})+\psi_y^\sigma(H_{ab},g_{ab})\right]+O(h_{\sigma}^4), \quad ~\sigma=c,f.
   \end{split}
\end{equation*}
\end{lemma}
\begin{proof}
   This result follows immediately from Lemma 2.4 in \cite{wangxp}.
\end{proof}
\begin{lemma} \label{lemma2.6}
   \cite{ZQF} Suppose that the mesh functions $u,v,w,S,T \in {\mathcal{V}}_h^\sigma$, satisfying
   \begin{equation*}
      \begin{split}
         v_{ab} = \delta_{xx}^\sigma u_{ab} - \frac{h_{\sigma}^2}{12} \delta_{xx}^\sigma v_{ab} + S_{ab}, \quad (x_a, y_b) \in \Omega_h^\sigma,\\
         w_{ab} = \delta_{yy}^\sigma u_{ab} - \frac{h_{\sigma}^2}{12} \delta_{yy}^\sigma w_{ab} + T_{ab}, \quad (x_a, y_b) \in \Omega_h^\sigma,
      \end{split}
   \end{equation*}
   thus we have 
   \begin{equation*}
      \langle v+w, u \rangle_\sigma \leq -|u|_{1,\sigma}^2 -\frac{h_{\sigma}^2}{18}\Big(\|v\|_\sigma^2 +\|w\|_\sigma^2\Big) + \frac{h_{\sigma}^2}{12}\langle v,S \rangle_\sigma + \frac{h_{\sigma}^2}{12}\langle w,T \rangle_\sigma + \langle u, S+T \rangle_\sigma.
   \end{equation*}
\end{lemma}

\begin{lemma} \label{lemma2.7}
   \cite{Sloan} (Discrete Gr\"{o}nwall's inequality) If $\{A_n\}_{n=0}^{\infty}$ and $\{B_n\}_{n=0}^{\infty}$ are two non-negative real sequences and $\{d_n\}_{n=0}^{\infty}$ is a non-descending and non-negative sequence satisfying
    $$A_n\leq B_n+\sum\limits_{j=0}^{n-1}d_j A_j,\quad n\geq 1,$$
    then one has
    $$A_n\leq B_n\exp \left( \sum\limits_{j=0}^{n-1}d_j\right) ,\quad n\geq 1.$$
   
\end{lemma}

\begin{remark}
   Throughout this paper, $\widetilde{C}$ denotes a generic constant whose value may vary from line to line but is independent of the temporal and spatial step sizes. 
\end{remark}

\begin{remark}
   In order to facilitate writing and reading, we simplify the notation in the fine grid as follows:
   $$
\langle \cdot,\cdot \rangle:=\langle \cdot,\cdot \rangle_f, \quad \|\cdot\|:=\|\cdot\|_f, \quad |\cdot|_{1}=|\cdot|_{1,f}.
   $$
\end{remark}

\section{Construction of the space-time two-grid compact difference scheme} \label{sec3}
In this section, we shall derive a space-time two-grid compact difference (ST-TGCD) scheme for the 2D viscous Burgers' equation \eqref{main_equation}. Throughout the paper, we only consider a period domain $\Omega=(0,L_1)\times(0,L_2)$.
Let us introduce two new variables $v=u_{xx}$ and $w=u_{yy}$, then the equation \eqref{main_equation} in a period can be rewritten as
\begin{align}
\label{eq3.1} &u_t+uu_x+uu_y=\lambda(v+w), \quad (x,y)\in \Omega, \quad 0<t\leq T,\\
\label{eq3.2} &v=u_{xx}, \quad w=u_{yy}, \quad (x,y)\in\Omega, \quad 0\leq t\leq T.
\end{align}
Define the mesh functions
$$
U_{ab}^d=u(x_a,y_b,t_d), \quad V_{ab}^d=v(x_a,y_b,t_d), \quad W_{ab}^d=w(x_a,y_b,t_d), \quad (x_a,y_b)\in\Omega_h^\sigma,~t_d\in\Omega_\tau^\sigma,~\sigma=c,f.
$$
Using Lemma \ref{lemma2.5}, we have
\begin{equation} \label{eq:uu_x_compact_scheme}
u(u_x+u_y)(x_a, y_b, t_d) = \psi_h^\sigma(U^d_{ab}, U^d_{ab}) - \frac{h_{\sigma}^2}{2}\left[\psi_x^\sigma({V}^d_{ab}, U^d_{ab})+\psi_y^\sigma({W}^d_{ab}, U^d_{ab})\right] + O(h_{\sigma}^4).
\end{equation}
By considering \eqref{eq3.1} at the point $(x_a,y_b,t_{d-\frac{1}{2}})$ in a single periodic domain, and combining Taylor expansion with \eqref{eq:uu_x_compact_scheme}, we have
\begin{equation} \label{eq3.4}
\begin{split}
 &\delta_{t}^\sigma U^{d-\frac{1}{2}}_{ab} + \psi_h^{\sigma}(U^{d-\frac{1}{2}}_{ab}, U^{d-\frac{1}{2}}_{ab}) - \frac{h_{\sigma}^{2}}{2}\left[\psi_x^\sigma({V}^{d-\frac{1}{2}}_{ab}, U^{d-\frac{1}{2}}_{ab})+\psi_y^\sigma({W}^{d-\frac{1}{2}}_{ab}, U^{d-\frac{1}{2}}_{ab})\right] \\ 
&\quad = \lambda(V^{d-\frac{1}{2}}_{ab} + W^{d-\frac{1}{2}}_{ab}) + (P_\sigma)^{d-\frac{1}{2}}_{ab}, \quad (x_{a}, y_{b})\in \Omega_h^\sigma ,\quad t_d\in\Omega_\tau^\sigma\setminus\{0\}, \quad \sigma=c,f,
\end{split}
\end{equation}
where $(P_\sigma)^{d-\frac{1}{2}}_{ab}$ is the truncation error term, which is of order $\mathcal{O}(h_{\sigma}^4 + \tau_\sigma^2)$. Then, considering \eqref{eq3.2} at the point $(x_a,y_b,t_{d})$, we have
\begin{align}
V_{ab}^d = \delta_{xx}^\sigma U_{ab}^d - \frac{h_{\sigma}^2}{12} \delta_{xx}^\sigma V_{ab}^d + (Q_\sigma)_{ab}^d, \quad (x_{a}, y_{b})\in \Omega_h^\sigma ,\quad t_d\in\Omega_\tau^\sigma, \quad \sigma=c,f, \label{eq3.5}\\ 
W_{ab}^d = \delta_{yy}^\sigma U_{ab}^d - \frac{h_{\sigma}^2}{12} \delta_{yy}^\sigma W_{ab}^d + (R_\sigma)_{ab}^d,\quad  (x_{a}, y_{b})\in \Omega_h^\sigma ,\quad t_d\in\Omega_\tau^\sigma, \quad \sigma=c,f,\label{eq3.6}
\end{align}
 where
 $$
|(Q_\sigma)_{ab}^d|=\mathcal{O}(h_{\sigma}^4), \quad |(R_\sigma)_{ab}^d|=\mathcal{O}(h_{\sigma}^4), \quad (x_{a}, y_{b})\in \Omega_h^\sigma ,\quad t_d\in\Omega_\tau^\sigma, \quad \sigma=c,f.
 $$
 Combining the initial condition \eqref{initial_condition} with boundary conditions \eqref{boundary_condition}, and omitting the small terms in \eqref{eq3.4}-\eqref{eq3.6}. Replacing $U_{ab}^{d}, V_{ab}^{d}, W_{ab}^{d}$ with their numerical approximations $u_{ab}^{d}, v_{ab}^{d}, w_{ab}^{d}$. Furthermore, if we take $\sigma=f$ and replace index $(x_a,y_b,t_d)$ with $(x_i,y_j,t_n)$, i.e., consider on the fine mesh, then the following nonlinear compact difference (NCD) scheme for the 2D viscous BE \eqref{main_equation}-\eqref{boundary_condition} is obtained:
\begin{align}
\notag &\delta_{t}^f u^{n-\frac{1}{2}}_{ij} + \psi_h^{f}(u^{n-\frac{1}{2}}_{ij}, u^{n-\frac{1}{2}}_{ij}) - \frac{h_{f}^{2}}{2}\left[\psi_x^f({v}^{n-\frac{1}{2}}_{ij}, u^{n-\frac{1}{2}}_{ij})+\psi_y^f({w}^{n-\frac{1}{2}}_{ij}, u^{n-\frac{1}{2}}_{ij})\right]\\ 
\label{eq3.7}   &\quad = \lambda(v^{n-\frac{1}{2}}_{ij} + w^{n-\frac{1}{2}}_{ij}), \quad (x_{i}, y_{j})\in \Omega_h^f ,\quad t_n\in\Omega_\tau^f\setminus\{0\},\\
\label{eq3.8} &v_{ij}^n = \delta_{xx}^f u_{ij}^n - \frac{h_{f}^2}{12} \delta_{xx}^f v_{ij}^n, \quad (x_{i}, y_{j})\in \Omega_h^f ,\quad t_n\in\Omega_\tau^f, \\
\label{eq3.9} &w_{ij}^n = \delta_{yy}^f u_{ij}^n - \frac{h_{f}^2}{12} \delta_{yy}^f w_{ij}^n, \quad (x_{i}, y_{j})\in \Omega_h^f ,\quad t_n\in\Omega_\tau^f,\\
\label{eq3.10} &u_{ij}^0 = u_0(x_i,y_j), \quad (x_{i}, y_{j})\in \Omega_h^f,\\
\notag &  u_{ij}^n = u_{i+M_1^f,j}^n, ~ u_{ij}^n = u_{i,j+M_2^f}^n, \quad v_{ij}^n = v_{i+M_1^f,j}^n, ~ v_{ij}^n = v_{i,j+M_2^f}^n,\\
\label{eq3.11} &w_{ij}^n = w_{i+M_1^f,j}^n, ~ w_{ij}^n = w_{i,j+M_2^f}^n, \quad (x_{i}, y_{j})\in \Omega_h^f, \quad t_n\in\Omega_\tau^f.
\end{align}
 \begin{remark}
The scheme \eqref{eq3.7}-\eqref{eq3.11} extends from the 1D case \cite{wangxp}. We solve the nonlinear system \eqref{eq3.7}-\eqref{eq3.11} by using the fixed-point iteration method. Nevertheless, excessively large values of $M_1^f, M_2^f, N^f$ severely degrade computational efficiency.
\end{remark}

Next, we construct the ST-TGCD scheme, which consists of three main steps as follows: 
\vskip 0.2cm
\textbf{Step 1.} First, we solve the above nonlinear system (NCD scheme) on the coarse grid $\Omega_h^c\times\Omega_\tau^c$ to get the coarse grid solution $(u_c)_{pq}^{r}$, namely, by solving the following system:
\begin{align}
 \label{eq3.12} &\delta_{t}^c (u_c)^{r-\frac{1}{2}}_{pq} + \psi_h^c\Big((u_c)^{r-\frac{1}{2}}_{pq}, (u_c)^{r-\frac{1}{2}}_{pq}\Big) - \frac{h_{c}^{2}}{2}\left[\psi_x^c\Big((v_c)^{r-\frac{1}{2}}_{pq}, (u_c)^{r-\frac{1}{2}}_{pq}\Big)+\psi_y^c\Big((w_c)^{r-\frac{1}{2}}_{pq}, (u_c)^{r-\frac{1}{2}}_{pq}\Big)\right]  \\ 
 \notag &\quad = \lambda\Big((v_c)^{r-\frac{1}{2}}_{pq} + (w_c)^{r-\frac{1}{2}}_{pq}\Big), \quad (x_{p}, y_{q})\in \Omega_h^c ,\quad t_r\in\Omega_\tau^c\setminus\{0\},\\
\label{eq3.13} &(v_c)_{pq}^r = \delta_{xx}^c (u_c)_{pq}^r - \frac{h_{c}^2}{12} \delta_{xx}^c (v_c)_{pq}^r, \quad (x_{p}, y_{q})\in \Omega_h^c ,\quad t_r\in\Omega_\tau^c, \\
\label{eq3.14} &(w_c)_{pq}^r = \delta_{yy}^c (u_c)_{pq}^r - \frac{h_{c}^2}{12} \delta_{yy}^c (w_c)_{pq}^r, \quad (x_{p}, y_{q})\in \Omega_h^c ,\quad t_r\in\Omega_\tau^c,\\
\label{eq3.15} &(u_c)_{pq}^0 = u_0(x_p,y_q), \quad (x_{p}, y_{q})\in \Omega_h^c,\\
\notag &  (u_c)_{pq}^r = (u_c)_{p+M_1^c,q}^r, ~ (u_c)_{pq}^r = (u_c)_{p,q+M_2^c}^r, \quad (v_c)_{pq}^r = (v_c)_{p+M_1^c,q}^r, ~ (v_c)_{pq}^r = (v_c)_{p,q+M_2^c}^r,\\
\label{eq3.16} &(w_c)_{pq}^r = (w_c)_{p+M_1^c,q}^r, ~ (w_c)_{pq}^r = (w_c)_{p,q+M_2^c}^r, \quad (x_{p}, y_{q})\in \Omega_h^c, \quad t_r\in\Omega_\tau^c.
\end{align}

\textbf{Step 2.} In temporal direction, we use the coarse grid solution $(u_c)_{pq}^{r}$ and Lagrange linear interpolation
\begin{equation}\label{eq3.17}
   \begin{split}
     (u_f)_{pk_h,qk_h}^{(r-1)k_\tau+s}&=\frac{t_{(r-1)k_\tau+s}-t_{rk_\tau}}{t_{(r-1)k_\tau}-t_{rk_\tau}}(u_c)_{pq}^{r-1}+\frac{t_{(r-1)k_\tau+s}-t_{(r-1)k_\tau}}{t_{rk_\tau}-t_{(r-1)k_\tau}}(u_c)_{pq}^{r}\\
     &=(1-\frac{s}{k_\tau})(u_c)_{pq}^{r-1}+\frac{s}{k_\tau}(u_c)_{pq}^{r},\quad 1\leq r \leq N^c,\quad 1\leq s \leq k_\tau-1,
   \end{split}
\end{equation}
where $k_h$ and $k_\tau$ are the spatial and temporal step-size ratios, respectively. And note that $(u_f)_{pk_h,qk_h}^{rk_\tau}=(u_c)_{pq}^r,~ 0\leq p \leq M_1^c,~0\leq q \leq M_2^c,~0\leq r \leq N^c$. Thus, $(u_f)_{pk_h,qk_h}^n~(0\leq n \leq N^f)$ is obtained. 

Then, in spatial direction, combining $(u_f)_{pk_h,qk_h}^n$ with cubic Lagrange interpolation to obtain the fine grid solution $(u_f)_{pk_h+l,qk_h+m}^n~(1\leq l,m \leq k_h-1)$, i.e.,
\begin{equation} \label{eq3.18}
\begin{split}
&(u_f)_{pk_h+l,qk_h+m}^n=\sum_{i=0}^3\sum_{j=0}^3 (u_f)_{(p+i)k_h,(q+j)k_h}^{n} {\mu_i(x_{pk_h+l})} {\omega_j(y_{qk_h+m})}, \\
&\quad x_{pk_h+l}\in (x_{pk_h},x_{(p+3)k_h}),~y_{qk_h+m}\in (y_{qk_h},y_{(q+3)k_h}),~0\leq p \leq M_1^c-1,~ 0\leq q \leq M_2^c-1,
\end{split}
\end{equation}
in which $\mu_i(x)$ and $\omega_j(y)$ are the cubic Lagrange interpolation basis functions defined as
$$
\mu_i(x)=\prod_{\substack{0\leq s \leq 3 \\ s\neq i}} \frac{x-x_{(p+s)k_h}}{x_{(p+i)k_h}-x_{(p+s)k_h}}, \quad \omega_j(y)=\prod_{\substack{0\leq s \leq 3 \\ s\neq j}} \frac{y-y_{(q+s)k_h}}{y_{(q+j)k_h}-y_{(q+s)k_h}}.
$$
Therefore, we can get the rough solution $(u_f)_{ij}^n$ on the fine grid for all $(x_i,y_j,t_n)\in\Omega_h^f \times \Omega_\tau^f$. Furthermore, we can compute the $(v_f)_{ij}^n$ and $(w_f)_{ij}^n$ by using the following formulas
\begin{align}
   \label{equ3.19} &(v_f)_{ij}^n = \delta_{xx}^f (u_f)_{ij}^n - \frac{h_f^2}{12}\delta_{xx}^f (v_f)_{ij}^n, \quad (x_i,y_j)\in\Omega_h^f,~t_n\in\Omega_\tau^f,\\
   \label{equ3.20} &(w_f)_{ij}^n = \delta_{yy}^f (u_f)_{ij}^n - \frac{h_f^2}{12}\delta_{yy}^f (w_f)_{ij}^n, \quad (x_i,y_j)\in\Omega_h^f,~t_n\in\Omega_\tau^f.
\end{align}

\vskip 0.2cm
\textbf{Step 3.} Finally, we correct the fine grid solution $(u_f)_{ij}^n$ to obtain higher precision numerical solution $u_{ij}^n$ by solving the following linear system:

\begin{align}
\notag  &\delta_{t}^f u^{n-\frac{1}{2}}_{ij} + \psi_h^f\left((u_f)^{n-\frac{1}{2}}_{ij}, u^{n-\frac{1}{2}}_{ij}\right) - \frac{h_{f}^{2}}{2}\left[\psi_x^f\left((v_f)^{n-\frac{1}{2}}_{ij}, u^{n-\frac{1}{2}}_{ij}\right)+\psi_y^f\left((w_f)^{n-\frac{1}{2}}_{ij}, u^{n-\frac{1}{2}}_{ij}\right)\right]  \\ 
\label{step3-1}  &\quad = \lambda(v^{n-\frac{1}{2}}_{ij} + w^{n-\frac{1}{2}}_{ij}), \quad (x_{i}, y_{j})\in \Omega_h^f ,\quad t_n\in\Omega_\tau^f\setminus\{0\},\\
\label{step3-2} &v_{ij}^n = \delta_{xx}^f u_{ij}^n - \frac{h_{f}^2}{12} \delta_{xx}^f v_{ij}^n, \quad (x_{i}, y_{j})\in \Omega_h^f ,\quad t_n\in\Omega_\tau^f, \\
\label{step3-3} &w_{ij}^n = \delta_{yy}^f u_{ij}^n - \frac{h_{f}^2}{12} \delta_{yy}^f w_{ij}^n, \quad (x_{i}, y_{j})\in \Omega_h^f ,\quad t_n\in\Omega_\tau^f,\\
\label{step3-4} &u_{ij}^0 = u_0(x_i,y_j), \quad (x_{i}, y_{j})\in \Omega_h^f,\\
\notag &  u_{ij}^n = u_{i+M_1^f,j}^n, ~ u_{ij}^n = u_{i,j+M_2^f}^n, \quad v_{ij}^n = v_{i+M_1^f,j}^n, ~ v_{ij}^n = v_{i,j+M_2^f}^n,\\
\label{step3-5} &w_{ij}^n = w_{i+M_1^f,j}^n, ~ w_{ij}^n = w_{i,j+M_2^f}^n, \quad (x_{i}, y_{j})\in \Omega_h^f, \quad t_n\in\Omega_\tau^f.
\end{align}
Note that $\{(u_f)_{ij}^n,(v_f)_{ij}^n,(w_f)_{ij}^n\}$ in \eqref{step3-1} are all known from Step 2. Thus, the scheme \eqref{step3-1}-\eqref{step3-5} is a linear system.
\section{Unique solvability and convergence of the ST-TGCD scheme} \label{sec4}
\subsection{Unique solvability}
  For the step 1 of the ST-TGCD scheme, we can employ Browder fixed-point theorem (cf. \cite{akrivis1992finite}) to prove the existence of the solution, and then use proof by contradiction combined with the energy method to demonstrate the uniqueness. Since the proof process is almost same as that for the 1D case, one can refer to Theorems 3.3 and 3.4 in \cite{wangxp}. To avoid excessive length of the article, the detailed proof process is omitted here. For the step 2, the unique solvability of $(u_f)_{ij}^n$ is straightforward, because we just use the linear and cubic Lagrange interpolation to get $(u_f)_{ij}^n$. Therefore, we only need to derive the unique solvability of the step 3.  

\begin{theorem}
  The third-step scheme \eqref{step3-1}-\eqref{step3-5} is uniquely solvable.
\end{theorem}
\begin{proof}
   To begin, we can know that $\{u^0,~v^0,~w^0\}$ have been determined uniquely by \eqref{step3-2}-\eqref{step3-4}. Now, we suppose that the solution $\{u_{ij}^{n-1}, v_{ij}^{n-1}, w_{ij}^{n-1}\}$ is known, and we shall demonstrate that the solution $\{u_{ij}^n, v_{ij}^n, w_{ij}^n\}$ of the linear scheme \eqref{step3-1}-\eqref{step3-3} exists. Let us consider the homogeneous system of \eqref{step3-1}-\eqref{step3-3} as follows
\begin{align}
   \notag  &\frac{1}{\tau_f} u^{n}_{ij} +\frac{1}{2} \psi_h^f\left((u_f)^{n-\frac{1}{2}}_{ij}, u^{n}_{ij}\right) - \frac{h_{f}^{2}}{4}\left[\psi_x^f\left((v_f)^{n-\frac{1}{2}}_{ij}, u^{n}_{ij}\right)+\psi_y^f\left((w_f)^{n-\frac{1}{2}}_{ij}, u^{n}_{ij}\right)\right]  \\ 
\label{eqt4.1}  &\quad = \frac{1}{2}\lambda(v^{n}_{ij} + w^{n}_{ij}), \quad (x_{i}, y_{j})\in \Omega_h^f ,\\
\label{eqt4.2} &v_{ij}^n = \delta_{xx}^f u_{ij}^n - \frac{h_{f}^2}{12} \delta_{xx}^f v_{ij}^n, \quad (x_{i}, y_{j})\in \Omega_h^f , \\
\label{eqt4.3} &w_{ij}^n = \delta_{yy}^f u_{ij}^n - \frac{h_{f}^2}{12} \delta_{yy}^f w_{ij}^n, \quad (x_{i}, y_{j})\in \Omega_h^f.
\end{align}
Taking the inner product of \eqref{eqt4.1} with $u^n$ and using Lemma \ref{lemma2.4}, we get
\begin{equation} \label{eqt4.4}
   \begin{split}
      \frac{1}{\tau_f}\|u^n\|^2  = \frac{\lambda}{2}\langle v^n+w^n, u^n \rangle.
   \end{split}
\end{equation}
Applying Lemma \ref{lemma2.6}, one obtains that
\begin{equation*}
   \begin{split}
      \frac{1}{\tau_f}\|u^n\|^2   +\frac{\lambda}{2}|u^n|_{1}^2 +\frac{\lambda h_{f}^2}{36}\Big(\|v^n\|^2 +\|w^n\|^2\Big)\leq 0.
   \end{split}
\end{equation*}
Then, it means that
\begin{equation*}
   \|u^n\|=0, \quad \|v^n\|=0, \quad \|w^n\|=0.
\end{equation*}
Thus, the homogeneous system \eqref{eqt4.1}-\eqref{eqt4.3} only has zeros solutions, which implies that $u^n,~v^n$ and $w^n$ are uniquely determined by \eqref{step3-1}-\eqref{step3-5}. 
\end{proof}

\subsection{Convergence}
In this section, we shall prove the convergence of the ST-TGCD scheme \eqref{eq3.12}-\eqref{step3-5}. The error bound for each individual step is independently derived.
\subsubsection{Error estimate for the step 1.}
      Firstly, we derive the convergence of coarse grid system \eqref{eq3.12}-\eqref{eq3.16}. Denote
\begin{equation*}
   \begin{split}
   &(e_c)_{pq}^{r}:=U_{pq}^{r}-(u_c)_{pq}^{r}, \quad (\rho_c)_{pq}^{r} :=V_{pq}^{r}-(v_c)_{pq}^{r}, \quad (\gamma_c)_{pq}^{r}:=W_{pq}^{r}-(w_c)_{pq}^{r}, \quad (x_p,y_q)\in \Omega_h^c,~t_r \in \Omega_\tau^c.\\
   \end{split}
\end{equation*}
\vskip 0.2mm
Subtracting \eqref{eq3.12}–\eqref{eq3.14} from \eqref{eq3.4}–\eqref{eq3.6} (note that we here take $\sigma=c$ and replace index $(x_a,y_b,t_d)$ with $(x_p,y_q,t_r)$ in \eqref{eq3.4}–\eqref{eq3.6}) yields the coarse-grid error equations
\begin{align}
 \notag &\delta_{t}^c (e_c)^{r-\frac{1}{2}}_{pq} + \psi_h^c\Big(U^{r-\frac{1}{2}}_{pq}, U^{r-\frac{1}{2}}_{pq}\Big) - \psi_h^c\Big((u_c)^{r-\frac{1}{2}}_{pq}, (u_c)^{r-\frac{1}{2}}_{pq}\Big) - \frac{h_{c}^{2}}{2}\left[\psi_x^c\Big(V^{r-\frac{1}{2}}_{pq}, U^{r-\frac{1}{2}}_{pq}\Big)\right.\\
 \notag &\quad \left.-~\psi_x^c\Big((v_c)^{r-\frac{1}{2}}_{pq}, (u_c)^{r-\frac{1}{2}}_{pq}\Big)+\psi_y^c\Big(W^{r-\frac{1}{2}}_{pq}, U^{r-\frac{1}{2}}_{pq}\Big)-\psi_y^c\Big((w_c)^{r-\frac{1}{2}}_{pq}, (u_c)^{r-\frac{1}{2}}_{pq}\Big)\right]  \\ 
 \label{equ4.5} &\quad = \lambda\Big((\rho_c)^{r-\frac{1}{2}}_{pq} + (\gamma_c)^{r-\frac{1}{2}}_{pq}\Big)+(P_c)^{r-\frac{1}{2}}_{pq}, \quad (x_{p}, y_{q})\in \Omega_h^c ,\quad t_r\in\Omega_\tau^c\setminus\{0\},\\
\label{eq4.2} &(\rho_c)_{pq}^r = \delta_{xx}^c (e_c)_{pq}^r - \frac{h_{c}^2}{12} \delta_{xx}^c(\rho_c)_{pq}^r+ (Q_c)_{pq}^r, \quad (x_{p}, y_{q})\in \Omega_h^c ,\quad t_r\in\Omega_\tau^c, \\
\label{eq4.3} &(\gamma_c)_{pq}^r = \delta_{yy}^c (e_c)_{pq}^r - \frac{h_{c}^2}{12} \delta_{yy}^c(\gamma_c)_{pq}^r +(R_c)_{pq}^r, \quad (x_{p}, y_{q})\in \Omega_h^c ,\quad t_r\in\Omega_\tau^c.
\end{align}

\begin{theorem} \label{Th4.2}
   Suppose that $\{(u_c)_{pq}^r,(v_c)_{pq}^r,(w_c)_{pq}^r~|~(x_p,y_q, t_r)\in \Omega_h^c\times \Omega_\tau^c\}$ is the solution of the coarse grid system \eqref{eq3.12}-\eqref{eq3.16}, $\{u(x,y,t),v(x,y,t), w(x,y,t)~|~ (x,y,t)\in \Omega\times[0,T]\}$ is the solution of \eqref{eq3.1}-\eqref{eq3.2} with \eqref{initial_condition}-\eqref{boundary_condition}. There exists a positive constant $\hat{c}$ that is independent of $h_c$ and $\tau_c$, when $\tau_c \leq \hat{c}$, then one can obtain the following error estimate
   \begin{equation} \label{equ4.8}
       \|e_c^r\|_{c} \leq \widetilde{C}(\tau_c^2+h_c^4), \qquad 1\leq r \leq N^c.
   \end{equation}
\end{theorem}
\begin{proof}
   The proof closely parallels that of the 1D case, details can be found in the Appendix of \cite{wangxp}.
\end{proof}

\subsubsection{Error estimate for the step 2.}
In what follows, we will derive the convergence of the interpolation scheme \eqref{eq3.17}-\eqref{eq3.18}. Denote
\begin{equation*}
   \begin{split}
   (e_f)_{ij}^{n}:=U_{ij}^{n}-(u_f)_{ij}^{n}, \quad (\rho_f)_{ij}^{n}:=V_{ij}^{n}-(v_f)_{ij}^{n}, \quad (\gamma_f)_{ij}^{n}:=W_{ij}^{n}-(w_f)_{ij}^{n}, \quad (x_i,y_j)\in \Omega_h^f,~t_n \in \Omega_\tau^f.
   \end{split}
\end{equation*}
\begin{theorem}
   Suppose $\{U_{ij}^n~|~(x_i,y_j)\in \Omega_h^f,~ t_n \in \Omega_\tau^f\}$ is the solution of \eqref{main_equation}-\eqref{boundary_condition}, $\{(u_f)_{ij}^n~|~(x_i,y_j)\in \Omega_h^f,~ t_n \in \Omega_\tau^f\}$ is the solution of \eqref{eq3.17}-\eqref{eq3.18}. If $u(x,y,t) \in C^{4,2}(\bar{\Omega}\times[0,T])$, then it holds that
   \begin{equation} \label{eqt4.9}
       \|e_f^n\| \leq \widetilde{C}(\tau_f^2+h_f^4), \qquad 1\leq n \leq N^f.
   \end{equation} 
\end{theorem}
\begin{proof}
 Assume that $u(x,y,t) \in C^{4,4,2}([0,L_1]\times[0,L_2]\times[0,T])$. For $1\leq r \leq N^c$ and $1\leq s \leq k_\tau-1$, combining with the error formula of Lagrange interpolation, one has
 \begin{equation} \label{equ4.10}
   \begin{split}
     U_{pk_h,qk_h}^{(r-1)k_\tau+s}=&(1-\frac{s}{k_\tau})U_{pk_h,qk_h}^{(r-1)k_\tau}+\frac{s}{k_\tau}U_{pk_h,qk_h}^{rk_\tau}\\
     &+\frac{u_{tt}(\xi )}{2}(t_{(r-1)k_\tau+s}-t_{(r-1)k_\tau})(t_{(r-1)k_\tau+s}-t_{rk_\tau}), \quad \xi \in (t_{(r-1)k_\tau},t_{rk_\tau}).
   \end{split}
 \end{equation}
  Subtracting (\ref{eq3.17}) from (\ref{equ4.10}), then we have
  $$ (e_f)_{pk_h,qk_h}^{(r-1)k_\tau+s}=(1-\frac{s}{k_\tau})(e_c)^{r-1}_{pq}+\frac{s}{k_\tau}(e_c)_{pq}^{r}+\frac{u_{tt}(\xi )}{2}(t_{(r-1)k_\tau+s}-t_{(r-1)k_\tau})(t_{(r-1)k_\tau+s}-t_{rk_\tau}). $$
  And using triangle inequality and \eqref{equ4.8}, we can obtain
  \begin{equation*}
   \begin{split}
     \left\| e_f^{(r-1)k_\tau+s}\right\|_{c} \leq& (1-\frac{s}{k_\tau})\left\| e_c^{r-1}\right\|_{c}+\frac{s}{k_\tau}\left\| e_c^{r}\right\|_{c}+\frac{\widetilde{C}}{2}\tau_c^2\\
     \leq & \widetilde{C}(\tau_c^2+h_c^4), \quad 1\leq r \leq N^c, \quad 1\leq s \leq k_\tau-1,
   \end{split}
  \end{equation*}
  in which $\left\| e_f^{(r-1)k_\tau+s}\right\|_{c}=\sqrt{h_c^2\sum_{p=1}^{M_1^c}\sum_{q=1}^{M_2^c} [(e_f)_{pk_h,qk_h}^{(r-1)k_\tau+s}]^2}$.
  Together with \eqref{equ4.8}. Thus, we have
  \begin{equation} \label{equ4.11}
   \| e_f^n\|_{c} \leq \widetilde{C}(\tau_c^2+h_c^4), \quad 0\leq n \leq N^f.
  \end{equation}
According to the 2D Lagrange cubic interpolation residual formula, we have
\begin{equation*} 
   \begin{split}
     &(e_f)_{pk_h+l,qk_h+m}^n = (U-u_f)_{pk_h+l,qk_h+m}^n = \sum_{i=0}^3\sum_{j=0}^3 (e_f)_{(p+i)k_h,(q+j)k_h}^{n} {\mu_i(x_{pk_h+l})} {\omega_j(y_{qk_h+m})}, \\
    & \quad +\frac{\varpi_x(x_{pk_h+l})}{4!}u^{(4,0)}(\alpha, y_{qk_h+m})+\frac{\varpi_y(y_{qk_h+m})}{4!}u^{(0,4)}(x_{pk_h+l}, \beta)- \frac{\varpi_{x}(x_{pk_h+l})\varpi_y(y_{qk_h+m})}{4!4!}u^{(4,4)}(\alpha', \beta'),
   \end{split}
\end{equation*}
where $\alpha, \alpha' \in (x_{pk_h},x_{(p+3)k_h});~\beta, \beta' \in (y_{qk_h},y_{(q+3)k_h})$; $u^{(k,l)}:=\frac{\partial^{k+l}u}{\partial^k x \partial^l y}$; $\varpi_x(x)$ and $\varpi_y(y)$ are defined as
$$\varpi_x(x):=\prod_{s=0}^3 (x-x_{(p+s)k_h}), \quad \varpi_y(y):=\prod_{s=0}^3 (y-y_{(q+s)k_h}).$$
Recall the definition of $\mu_i(x), \omega_j(y)$, for any $1 \leq l,m \leq k_h-1$, it holds that
\begin{equation} \label{equ4.11-1}
   |\mu_i(x_{pk_h+l})|<1, \quad |\omega_j(y_{qk_h+m})|<1.
\end{equation}
And there exists the estimates for $\varpi_{x}(x_{pk_h+l})$ and $\varpi_y(y_{qk_h+m})$, i.e.,
\begin{equation} \label{equ4.11-2}
   \varpi_{x}(x_{pk_h+l})\leq 6h_c^4, \quad \varpi_y(y_{qk_h+m})\leq 6h_c^4, \quad 1\leq l,m \leq k_h-1.
\end{equation}
Moreover, note that $u^{(4,0)}(\alpha, y_{qk_h+m}), u^{(0,4)}(x_{pk_h+l}, \beta)$ and $u^{(4,4)}(\alpha', \beta')$ are bounded by $\widetilde{C}$, and combine \eqref{equ4.11-1} with \eqref{equ4.11-2}, then, for any $0\leq p \leq M_1^c-1,\,0\leq q \leq M_2^c-1,\,0 \leq n \leq N^f$, we have
\begin{equation*} 
     \begin{split}
     |(e_f)_{pk_h+l,qk_h+m}^n| &\leq  \sum_{i=0}^3\sum_{j=0}^3 |(e_f)_{(p+i)k_h,(q+j)k_h}^{n}|+\frac{6h_c^4}{4!}\widetilde{C} +\frac{6h_c^4}{4!}\widetilde{C}+ \frac{36h_c^8}{4!4!}\widetilde{C}\\
     &\leq \sum_{i=0}^3\sum_{j=0}^3 |(e_f)_{(p+i)k_h,(q+j)k_h}^{n}|+ \widetilde{C}h_c^4, \quad 1 \leq l,m \leq k_h-1.
   \end{split}
\end{equation*}
Using Cauchy-Schwarz inequality in sequence form, namely, $(\sum_{i=1}^{n}a_ib_i)^2\leq (\sum_{i=1}^{n}a_i^2)(\sum_{i=1}^{n}b_i^2)$, we have
\begin{equation}\label{equ4.12}
\begin{split}
   |(e_f)_{pk_h+l,qk_h+m}^n|^2 \leq 17\sum_{i=0}^3\sum_{j=0}^3 |(e_f)_{(p+i)k_h,(q+j)k_h}^{n}|^2 + 17\widetilde{C}^2h_c^8.
\end{split}
\end{equation}
Meanwhile, we also have
\begin{equation} \label{equ4.13}
     \begin{split}
    & |(e_f)_{pk_h,qk_h+m}^n| \leq \frac{6h_c^4}{4!}\widetilde{C} \leq \widetilde{C}h_c^4, \quad 1 \leq m \leq k_h-1,\\
   &|(e_f)_{pk_h+l,qk_h}^n| \leq \frac{6h_c^4}{4!}\widetilde{C} \leq \widetilde{C}h_c^4, \quad 1 \leq l \leq k_h-1.
   \end{split}
\end{equation}

Thus, by using \eqref{equ4.11}, \eqref{equ4.12} and \eqref{equ4.13}, we obtain the convergence order in $L^2$-norm for the second step as follows
\begin{equation} \label{eq4.10}
   \begin{split}
       &\|e_f^n\|^2=h_f^2\sum_{i=1}^{M_1^f}\sum_{j=1}^{M_2^f}|(e_f)_{ij}^n|^2=h_f^2\sum_{p=0}^{M_1^c-1}\sum_{q=0}^{M_2^c-1}\sum_{l=0}^{k_h-1}\sum_{m=0}^{k_h-1}|(e_f)_{pk_h+l,qk_h+m}^n|^2\\
       &=h_f^2\sum_{p=0}^{M_1^c-1}\sum_{q=0}^{M_2^c-1}\left(|(e_f)_{pk_h,qk_h}^n|^2+\sum_{l=1}^{k_h-1}\sum_{m=1}^{k_h-1}|(e_f)_{pk_h+l,qk_h+m}^n|^2+\sum_{l=1}^{k_h-1}|(e_f)_{pk_h+l,qk_h}^n|^2+\sum_{m=1}^{k_h-1}|(e_f)_{pk_h,qk_h+m}^n|^2\right)\\
       & \leq \frac{1}{k_h^2}\|e_f^n\|_c^2+\frac{17\cdot 16(k_h-1)^2}{k_h^2}\|e_f^n\|_c^2+17L_1L_2\widetilde{C}^2(k_h-1)^2h_c^8+\frac{L_1L_2(k_h-1)}{k_h^2}\widetilde{C}^2h_c^8+\frac{L_1L_2(k_h-1)}{k_h^2}\widetilde{C}^2h_c^8\\
       & \leq \widetilde{C}(\tau_c^2+h_c^4)^2+\widetilde{C}h_c^8 \leq \widetilde{C}\max(k_\tau^4,k_h^8)(\tau_f^2+h_f^4)^2+\widetilde{C}k_h^8h_f^8 \leq \widetilde{C} (\tau_f^2+h_f^4)^2, \quad 0\leq n \leq N^f,
   \end{split}
\end{equation}
where the equations hold $\|e_f^n\|_c^2=h_c^2\sum_{p=0}^{M_1^c-1}\sum_{q=0}^{M_2^c-1}|(e_f)_{(p+i)k_h,(q+j)k_h}^{n}|^2~ (0\leq i,j \leq 3)$ by the help of periodic boundary. And we used the fact that $h_c=k_hh_f$ and $\tau_c=k_\tau\tau_f$ in \eqref{eq4.10}.
\end{proof}

We take $\sigma=f$ and replace index $(x_a,y_b,t_d)$ with $(x_i,y_j,t_n)$ in \eqref{eq3.5}-\eqref{eq3.6}. And then, subtracting \eqref{equ3.19}-\eqref{equ3.20} from \eqref{eq3.5}-\eqref{eq3.6}, respectively. We have
\begin{align}
   \label{equ4.15} &(\rho_f)_{ij}^n = \delta_{xx}^f (e_f)_{ij}^n - \frac{h_f^2}{12}\delta_{xx}^f (\rho_f)_{ij}^n+(Q_f)_{ij}^n, \quad (x_i,y_j)\in\Omega_h^f,~t_n\in\Omega_\tau^f,\\
   \label{equ4.16} &(\gamma_f)_{ij}^n = \delta_{yy}^f (e_f)_{ij}^n - \frac{h_f^2}{12}\delta_{yy}^f (\gamma_f)_{ij}^n+(R_f)_{ij}^n, \quad (x_i,y_j)\in\Omega_h^f,~t_n\in\Omega_\tau^f.
\end{align}
Rearranging \eqref{equ4.15}, we have
\begin{equation*}
   (\rho_f)_{ij}^n=\left(I+\frac{h_f^2}{12}\delta_{xx}^f\right)^{-1}\delta_{xx}^f (e_f)_{ij}^n+\left(I+\frac{h_f^2}{12}\delta_{xx}^f\right)^{-1}(Q_f)_{ij}^n,
\end{equation*}
where $I$ is the identity operator. Define $A:=I+\frac{h_f^2}{12}\delta_{xx}^f$, then using discrete Fourier transform to yield its eigenvalues $(\lambda_A)_j=1-\frac{1}{3}\sin^2(\frac{\pi j}{M_1^f}) $, see [Subsection 3.1, \citenum{gray2006toeplitz}]. Thus, we have the spectral norm $\|(I+\frac{h_f^2}{12}\delta_{xx}^f)^{-1}\|_{\text{spec}}\leq3/2$. Combining with Lemma \ref{lemma2.3}, then we obtain
\begin{equation} \label{equ4.17}
   \begin{split}
      \|\rho_f^n\| &\leq \|(I+\frac{h_f^2}{12}\delta_{xx}^f)^{-1}\|_{\text{spec}}\|\delta_{xx}^f e_f^n\|+\|(I+\frac{h_f^2}{12}\delta_{xx}^f)^{-1}\|_{\text{spec}}\|Q_f^n\|\\
      &\leq \frac{3}{2}\|\delta_{xx}^f e_f^n\|+\frac{3}{2}\|Q_f^n\| \leq \frac{6}{h_f^2}\|e_f^n\|+ \frac{3}{2}\|Q_f^n\|, \quad 0\leq n \leq N^f.
   \end{split}
\end{equation}
As the same way, we also have
\begin{equation} \label{equ4.18}
   \|\gamma_f^n\| \leq \frac{6}{h_f^2}\|e_f^n\|+ \frac{3}{2}\|R_f^n\|, \quad 0\leq n \leq N^f.
\end{equation}

 \subsubsection{Error estimate for the step 3.}
  \vskip 0.2mm
  Finally, we establish the convergence of the third-step scheme (\ref{step3-1})–(\ref{step3-5}). To this end, we also first introduce the errors and the requisite constant as follows
\begin{equation*}
   \begin{aligned}
     &e_{ij}^n:=U_{ij}^n-u_{ij}^n, \quad \rho_{ij}^n :=V_{ij}^n-v_{ij}^n, \quad \gamma_{ij}^n :=W_{ij}^n-w_{ij}^n, \qquad (x_i,y_j)\in \Omega_h^f,\quad t_n \in \Omega_\tau^f,\\
     &c_0:=\max \limits_{(x,y,t)\in \bar{\Omega}\times[0,T]}\left\{ |u(x,y,t)|,|u_{x}(x,y,t)+u_{y}(x,y,t)|,|u_{xx}(x,y,t)|,|u_{yy}(x,y,t)|\right\}.
   \end{aligned}
 \end{equation*}
  \vskip 0.2mm
  Subtracting \eqref{step3-1}-\eqref{step3-5} from \eqref{eq3.4}-\eqref{eq3.6} respectively. Note that taking $\sigma=f$ and replacing index $(x_a,y_b,t_d)$ with $(x_i,y_j,t_n)$ in \eqref{eq3.4}-\eqref{eq3.6}. Then the error system as follows
 \begin{align}
 \notag &\delta_{t}^f e^{n-\frac{1}{2}}_{ij} + \psi_h^f\Big(U^{n-\frac{1}{2}}_{ij}, U^{n-\frac{1}{2}}_{ij}\Big) - \psi_h^f\Big((u_f)^{n-\frac{1}{2}}_{ij}, u^{n-\frac{1}{2}}_{ij}\Big) - \frac{h_{f}^{2}}{2}\left[\psi_x^f\Big(V^{n-\frac{1}{2}}_{ij}, U^{n-\frac{1}{2}}_{ij}\Big)\right.\\
 \notag &\quad \left.-~\psi_x^f\Big((v_f)^{n-\frac{1}{2}}_{ij}, u^{n-\frac{1}{2}}_{ij}\Big)+\psi_y^f\Big(W^{n-\frac{1}{2}}_{ij}, U^{n-\frac{1}{2}}_{ij}\Big)-\psi_y^f\Big((w_f)^{n-\frac{1}{2}}_{ij}, u^{n-\frac{1}{2}}_{ij}\Big)\right]  \\ 
 \label{equ4.19} &\quad = \lambda\Big(\rho^{n-\frac{1}{2}}_{ij} + \gamma^{n-\frac{1}{2}}_{ij}\Big)+(P_f)_{ij}^n, \quad (x_{i}, y_{j})\in \Omega_h^f ,\quad t_n\in\Omega_\tau^f\setminus\{0\},\\
\label{equ4.20} &\rho_{ij}^n = \delta_{xx}^f e_{ij}^n - \frac{h_{f}^2}{12} \delta_{xx}^f \rho_{ij}^n+(Q_f)_{ij}^n, \quad (x_{i}, y_{j})\in \Omega_h^f ,\quad t_n\in\Omega_\tau^f, \\
\label{equ4.21} &\gamma_{ij}^n = \delta_{yy}^f e_{ij}^n - \frac{h_{f}^2}{12} \delta_{yy}^f \gamma_{ij}^n+(R_f)_{ij}^n, \quad (x_{i}, y_{j})\in \Omega_h^f ,\quad t_n\in\Omega_\tau^f.
\end{align}
  \begin{theorem} \label{Th4.4}
   Suppose that $\{ u_{ij}^{n},v_{ij}^{n}, w_{ij}^{n}~|~(x_i,y_j,t_n)\in \Omega_h^f\times\Omega_\tau^f\}$ is the solution of fine-grid linear system \eqref{step3-1}-\eqref{step3-5} and $\{ u(x,y,t),v(x,y,t), w(x,y,t) \}$ is the solution of \eqref{eq3.1}-\eqref{eq3.2} and \eqref{initial_condition}-\eqref{boundary_condition}. When  $\tau_f\leq 2/3c_1$, where $c_1$ is a positive constant defined in \eqref{equ4.25}. Then we can obtain
   \begin{equation} \label{equ4.22}
     \|e^{n}\| \leq \widetilde{C}(\tau_f^2+h_f^4), \qquad 1\leq n \leq N^f.
   \end{equation}
  \end{theorem}
  \begin{proof}
 Taking an inner product of (\ref{equ4.19}) with $e^{n-\frac{1}{2}}$, we have
 \begin{equation} \label{equ4.23}
   \begin{split}
     &\left\langle \delta_t^f e^{n-\frac{1}{2}}~,~e^{n-\frac{1}{2}} \right\rangle+\left\langle \psi_h^f(U^{n-\frac{1}{2}},U^{n-\frac{1}{2}})-\psi_h^f(u_f^{n-\frac{1}{2}},u^{n-\frac{1}{2}})~,~e^{n-\frac{1}{2}} \right\rangle\\
     &-\frac{h_f^2}{2}\left\langle \psi_x^f(V^{n-\frac{1}{2}},U^{n-\frac{1}{2}})-\psi_x^f(v_f^{n-\frac{1}{2}},u^{n-\frac{1}{2}})~,~e^{n-\frac{1}{2}} \right\rangle-\frac{h_f^2}{2}\left\langle \psi_y^f(W^{n-\frac{1}{2}},U^{n-\frac{1}{2}})-\psi_y^f(w_f^{n-\frac{1}{2}},u^{n-\frac{1}{2}})~,~e^{n-\frac{1}{2}} \right\rangle\\
     &=\lambda\left\langle \rho^{n-\frac{1}{2}}+\gamma^{n-\frac{1}{2}}~,~e^{n-\frac{1}{2}} \right\rangle + \left\langle P_f^{n-\frac{1}{2}}~,~e^{n-\frac{1}{2}} \right\rangle, \quad 1 \leq n \leq N^f.
   \end{split}
 \end{equation}
 In what follows, each term of (\ref{equ4.23}) is analyzed sequentially.\\
 \begin{equation} \label{equ4.24}
  \textbf{(\Rmnum{1})}:=\left\langle \delta_t e^{n-\frac{1}{2}}~,~e^{n-\frac{1}{2}} \right\rangle =\frac{1}{2\tau_f}(\| e^n \|^2-\| e^{n-1} \|^2).
 \end{equation}
  For the second term, with the aid of Lemmas \ref{lem2.1}-\ref{lemma2.4}, Cauchy–Schwarz inequality, and Young inequality, and combining $u_f^{n-\frac{1}{2}}=U^{n-\frac{1}{2}}-e_f^{n-\frac{1}{2}}$ with $u^{n-\frac{1}{2}}=U^{n-\frac{1}{2}}-e^{n-\frac{1}{2}} $, we have
 \begin{equation*}
   \begin{split}
    \textbf{(\Rmnum{2})}:= &-\left\langle \psi_h^f(U^{n-\frac{1}{2}},U^{n-\frac{1}{2}})-\psi_h^f(u_f^{n-\frac{1}{2}},u^{n-\frac{1}{2}})~,~e^{n-\frac{1}{2}} \right\rangle\\
     =&-\left\langle \psi_h^f(U^{n-\frac{1}{2}}~,~e^{n-\frac{1}{2}})+\psi_h^f(e_f^{n-\frac{1}{2}},U^{n-\frac{1}{2}})-\psi_h^f(e_f^{n-\frac{1}{2}}~,~e^{n-\frac{1}{2}})~,~e^{n-\frac{1}{2}} \right\rangle\\
     =&-\left\langle \psi_h^f(e_f^{n-\frac{1}{2}},U^{n-\frac{1}{2}})~,~e^{n-\frac{1}{2}}\right\rangle= -\frac{1}{3}\left[\left\langle e_f^{n-\frac{1}{2}}\widehat{\delta}_h^f U^{n-\frac{1}{2}}~,~e^{n-\frac{1}{2}} \right\rangle+\left\langle \widehat{\delta}_h^f(e_f^{n-\frac{1}{2}} U^{n-\frac{1}{2}})~,~e^{n-\frac{1}{2}} \right\rangle\right]\\
     =&-\frac{1}{3}\left\langle e_f^{n-\frac{1}{2}}\widehat{\delta}_h^f U^{n-\frac{1}{2}}~,~e^{n-\frac{1}{2}} \right\rangle+\frac{1}{3}\left\langle e_f^{n-\frac{1}{2}} U^{n-\frac{1}{2}}, \widehat{\delta}_h^f e^{n-\frac{1}{2}} \right\rangle\\
     \leq &\frac{c_0}{3}\|e_f^{n-\frac{1}{2}}\|\|e^{n-\frac{1}{2}}\|+\frac{\sqrt{2}c_0}{3}\|e_f^{n-\frac{1}{2}}\||e^{n-\frac{1}{2}}|_{1}\\
     \leq &\frac{c_0}{6}\|e_f^{n-\frac{1}{2}}\|^2+\frac{c_0}{6}\|e^{n-\frac{1}{2}}\|^2+\frac{c_0^2}{9\lambda}\|e_f^{n-\frac{1}{2}}\|^2+\frac{\lambda}{2}|e^{n-\frac{1}{2}}|_{1}^2.
   \end{split}
 \end{equation*}
 Similarly, using $v_f^{n-\frac{1}{2}}=V^{n-\frac{1}{2}}-\rho_f^{n-\frac{1}{2}}$ and $u^{n-\frac{1}{2}}=U^{n-\frac{1}{2}}-e^{n-\frac{1}{2}} $, we can estimate the third term as follows
 \begin{equation*}
   \begin{split}
    \textbf{(\Rmnum{3})}:=\; &\frac{h_f^2}{2}\left\langle \psi_x^f(V^{n-\frac{1}{2}},U^{n-\frac{1}{2}})-\psi_x^f(v_f^{n-\frac{1}{2}},u^{n-\frac{1}{2}})~,~e^{n-\frac{1}{2}} \right\rangle\\
     =\;&\frac{h_f^2}{2}\left\langle \psi_x^f(V^{n-\frac{1}{2}},e^{n-\frac{1}{2}})+\psi_x^f(\rho_f^{n-\frac{1}{2}},U^{n-\frac{1}{2}})-\psi_x^f(\rho_f^{n-\frac{1}{2}},e^{n-\frac{1}{2}}) ~,~e^{n-\frac{1}{2}}\right\rangle\\
       =\;& \frac{h_f^2}{2}\left\langle \psi_x^f(\rho_f^{n-\frac{1}{2}},U^{n-\frac{1}{2}})~,~e^{n-\frac{1}{2}}\right\rangle= \frac{h_f^2}{6}\left[\left\langle \rho_f^{n-\frac{1}{2}}\widehat{\delta}_x^f U^{n-\frac{1}{2}}~,~e^{n-\frac{1}{2}} \right\rangle-\left\langle \rho_f^{n-\frac{1}{2}} U^{n-\frac{1}{2}}~,~\widehat{\delta}_x^f e^{n-\frac{1}{2}} \right\rangle\right]\\
       \leq &\frac{h_f^2c_0}{6}\| \rho_f^{n-\frac{1}{2}}\|\|e^{n-\frac{1}{2}}\|+\frac{h_f^2c_0}{6}\|\rho_f^{n-\frac{1}{2}}\|\|\widehat{\delta}_x^f e^{n-\frac{1}{2}}\|.
\end{split}
\end{equation*}
Applying \eqref{equ4.17} and Lemma \ref{lemma2.3} to the above inequality, we have
\begin{equation*}
   \begin{split}
    \textbf{(\Rmnum{3})} &\leq \frac{h_f^2c_0}{6}\left(\frac{6}{h_f^2}\|e_f^{n-\frac{1}{2}}\|+ \frac{3}{2}\|Q_f^{n-\frac{1}{2}}\|\right)\|e^{n-\frac{1}{2}}\|+\frac{h_f^2c_0}{6}\left(\frac{6}{h_f^2}\|e_f^{n-\frac{1}{2}}\|+ \frac{3}{2}\|Q_f^{n-\frac{1}{2}}\|\right)\|\widehat{\delta}_x^f e^{n-\frac{1}{2}}\|\\
      &\leq c_0 \|e_f^{n-\frac{1}{2}}\| \|e^{n-\frac{1}{2}}\|+\frac{h_f^2c_0}{4}\|Q_f^{n-\frac{1}{2}}\|\|e^{n-\frac{1}{2}}\|+ c_0 \|e_f^{n-\frac{1}{2}}\| \|\widehat{\delta}_x^f e^{n-\frac{1}{2}}\|+\frac{h_f c_0}{2}\|Q_f^{n-\frac{1}{2}}\| \| e^{n-\frac{1}{2}}\|.
   \end{split}
\end{equation*}
 Then, by the same analysis, we can estimate the fourth term as follows
 \begin{equation*}
   \begin{split}
    \textbf{(\Rmnum{4})}:=& ~ \frac{h_f^2}{2}\left\langle \psi_y^f(W^{n-\frac{1}{2}},U^{n-\frac{1}{2}})-\psi_y^f(w^{n-\frac{1}{2}},u_f^{n-\frac{1}{2}})~,~e^{n-\frac{1}{2}} \right\rangle\\
       \leq &\frac{h_f^2c_0}{6}\| \gamma_f^{n-\frac{1}{2}}\|\|e^{n-\frac{1}{2}}\|+\frac{h_f^2c_0}{6}\|\gamma_f^{n-\frac{1}{2}}\|\|\widehat{\delta}_y^f e^{n-\frac{1}{2}}\|\\
       \leq & c_0 \|e_f^{n-\frac{1}{2}}\| \|e^{n-\frac{1}{2}}\|+\left(\frac{h_f^2c_0}{4}+\frac{h_f c_0}{2}\right)\|R_f^{n-\frac{1}{2}}\|\|e^{n-\frac{1}{2}}\|+ c_0 \|e_f^{n-\frac{1}{2}}\| \|\widehat{\delta}_y^f e^{n-\frac{1}{2}}\|.
\end{split}
\end{equation*}
     Furthermore, combining \textbf{(\Rmnum{3})} with \textbf{(\Rmnum{4})}, and using Young inequality, we have
 \begin{equation*}
   \begin{split}
    &\textbf{(\Rmnum{3})}+\textbf{(\Rmnum{4})}\\
     & \leq 2c_0 \|e_f^{n-\frac{1}{2}}\| \|e^{n-\frac{1}{2}}\|+\left(\frac{h_f^2c_0}{4}+\frac{h_f c_0}{2}\right)\Big(\|Q_f^{n-\frac{1}{2}}\|+\|R_f^{n-\frac{1}{2}}\|\Big)\|e^{n-\frac{1}{2}}\| + c_0 \|e_f^{n-\frac{1}{2}}\| \Big(\|\widehat{\delta}_x^f e^{n-\frac{1}{2}}\|+\|\widehat{\delta}_y^f e^{n-\frac{1}{2}}\|\Big)\\
     &\leq 2c_0 \|e_f^{n-\frac{1}{2}}\| \|e^{n-\frac{1}{2}}\|+\left(\frac{h_f^2c_0}{4}+\frac{h_f c_0}{2}\right)\Big(\|Q_f^{n-\frac{1}{2}}\|+\|R_f^{n-\frac{1}{2}}\|\Big)\|e^{n-\frac{1}{2}}\|+\sqrt{2}c_0 \|e_f^{n-\frac{1}{2}}\| |e^{n-\frac{1}{2}}|_{1}\\
     &\leq c_0 \|e_f^{n-\frac{1}{2}}\|^2+c_0 \|e^{n-\frac{1}{2}}\|^2+\left(\frac{h_f^2c_0}{8}+\frac{h_f c_0}{4}\right)\left(\|Q_f^{n-\frac{1}{2}}\|^2+\|R_f^{n-\frac{1}{2}}\|^2+2\|e^{n-\frac{1}{2}}\|^2\right)+\frac{\lambda}{2} |e^{n-\frac{1}{2}}|_{1}^2+\frac{c_0^2}{\lambda}\|e_f^{n-\frac{1}{2}}\|^2\\
     & = ~ \left(c_0+\frac{c_0^2}{\lambda}\right)\|e_f^{n-\frac{1}{2}}\|^2+\left(\frac{h_f^2c_0}{4}+\frac{h_f c_0}{2}\right)\|e^{n-\frac{1}{2}}\|^2+\frac{\lambda}{2} |e^{n-\frac{1}{2}}|_{1}^2 +\left(\frac{h_f^2c_0}{8}+\frac{h_f c_0}{4}\right)\Big(\|Q_f^{n-\frac{1}{2}}\|^2+\|R_f^{n-\frac{1}{2}}\|^2\Big).
   \end{split}
 \end{equation*}
 By invoking Lemma \ref{lemma2.6}, we derive the estimate for the fifth term in (\ref{equ4.23}) as follows
 \begin{equation*}
   \begin{split}
    \textbf{(\Rmnum{5})}&:= \lambda\left\langle \rho^{n-\frac{1}{2}}+\gamma^{n-\frac{1}{2}}~,~e^{n-\frac{1}{2}} \right\rangle \\
    &\leq \lambda\left( -|e^{n-\frac{1}{2}} |_{1}^2-\frac{h_f^2}{18}\Big(\|\rho^{n-\frac{1}{2}} \|^2+\|\gamma^{n-\frac{1}{2}} \|^2\Big)+\frac{h_f^2}{12}\Big(\langle Q_f^{n-\frac{1}{2}},\rho^{n-\frac{1}{2}}\rangle+\langle R_f^{n-\frac{1}{2}},\gamma^{n-\frac{1}{2}}\rangle\Big) +\langle Q_f^{n-\frac{1}{2}}+R_f^{n-\frac{1}{2}}~,~e^{n-\frac{1}{2}}\rangle\right)\\
     &\leq \lambda\left( -|e^{n-\frac{1}{2}} |_{1}^2+\Big(\frac{h_f^2}{32}+\frac{1}{2}\Big)\Big(\|Q_f^{n-\frac{1}{2}} \|^2+\|R_f^{n-\frac{1}{2}} \|^2\Big)+\|e^{n-\frac{1}{2}} \|^2\right),
   \end{split}
 \end{equation*}
 where we have used the Cauchy-Schwarz inequality and Young inequality to estimate 
 \begin{equation*}\begin{cases}
\begin{split}
       &\langle Q_f^{n-\frac{1}{2}},\rho^{n-\frac{1}{2}}\rangle\leq \|Q_f^{n-\frac{1}{2}} \|\|\rho^{n-\frac{1}{2}} \| \leq \frac{3}{8}\|Q_f^{n-\frac{1}{2}} \|^2+\frac{2}{3}\|\rho^{n-\frac{1}{2}} \|^2;\\
       &\langle R_f^{n-\frac{1}{2}},\gamma^{n-\frac{1}{2}}\rangle \leq \|R_f^{n-\frac{1}{2}} \|\|\gamma^{n-\frac{1}{2}} \| \leq \frac{3}{8}\|R_f^{n-\frac{1}{2}} \|^2+\frac{2}{3}\|\gamma^{n-\frac{1}{2}} \|^2;\\
      &\langle Q_f^{n-\frac{1}{2}}+R_f^{n-\frac{1}{2}}~,~e^{n-\frac{1}{2}}\rangle \leq \|Q_f^{n-\frac{1}{2}}\|\|e^{n-\frac{1}{2}} \|+\|R_f^{n-\frac{1}{2}}\|\|e^{n-\frac{1}{2}} \| \leq \frac{1}{2}\|Q_f^{n-\frac{1}{2}}\|^2+\frac{1}{2}\|R_f^{n-\frac{1}{2}}\|^2+\|e^{n-\frac{1}{2}} \|^2.
   \end{split}
 \end{cases}
 \end{equation*}
As the last term in (\ref{equ4.23}), we have
 \begin{equation*}
  \textbf{(\Rmnum{6})}:= \langle P_f^{n-\frac{1}{2}}~,~e^{n-\frac{1}{2}}\rangle \leq \frac{1}{2}\| P_f^{n-\frac{1}{2}}\|^2+\frac{1}{2}\| e^{n-\frac{1}{2}}\|^2.
 \end{equation*}
 Substituting the above results of the estimates for terms \textbf{(\Rmnum{1})}–\textbf{(\Rmnum{6})} into (\ref{equ4.23}) yields
 \begin{equation*} 
   \begin{split}
     \frac{1}{2\tau_f}\left( \|e^n \|^2-\|e^{(n-1)} \|^2 \right)\leq c_1\| e^{n-\frac{1}{2}}\|^2+c_2\| e_f^{n-\frac{1}{2}}\|^2+c_3\|Q_f^{n-\frac{1}{2}}\|^2+c_3\|R_f^{n-\frac{1}{2}}\|^2+\frac{1}{2}\| P_f^{n-\frac{1}{2}}\|^2,
   \end{split}
 \end{equation*}
   where
    \begin{equation} \label{equ4.25}
   c_1=\frac{c_0}{6}+\frac{h_f^2c_0}{4}+\frac{h_f c_0}{2}+\lambda+\frac{1}{2}, \quad c_2=\frac{7c_0}{6}+\frac{10c_0^2}{9\lambda}, \quad c_3=\lambda\left(\frac{h_f^2}{32}+\frac{1}{2}\right)+\frac{h_f^2c_0}{8}+\frac{h_f c_0}{4}.
   \end{equation}
 Recall the truncation errors $(P_f)_{ij}^{n-\frac{1}{2}}, (Q_f)_{ij}^{n}$ and $(R_f)_{ij}^{n}$ in \eqref{eq3.4}-\eqref{eq3.6} with $\sigma=f$. Meanwhile, combining the estimation result of the second step \eqref{eq4.10}, we can obtain
 $$\frac{1}{2\tau_f}\left( \|e^n \|^2-\|e^{(n-1)} \|^2 \right)\leq \frac{c_1}{2}\left(\|e^n \|^2+\|e^{(n-1)} \|^2\right)+\widetilde{C}(\tau_f^2+h_f^4)^2, $$
 that is
 \begin{equation} \label{equ4.26}
    \|e^n \|^2-\|e^{(n-1)} \|^2 \leq c_1\tau_f \left(\|e^n \|^2+\|e^{(n-1)} \|^2\right)+2\widetilde{C}\tau_f(\tau_f^2+h_f^4)^2.
 \end{equation}
 For any integer $1\leq m \leq N^f$, summing the above inequality over $n$ from 1 to $m$, we get
  $$ \|e^{m} \|^2 \leq c_1\tau_f\|e^{m} \|^2+\sum \limits_{n=1}^{m-1}2c_1\tau_f\|e^n \|^2+2\widetilde{C}m\tau_f(\tau_f^2+h_f^4)^2. $$
  When $c_1\tau_f<2/3$, apply the discrete Gr\"{o}nwall inequality, namely, Lemma \ref{lemma2.7}, we can obtain
  \begin{equation*}
   \begin{split}
     \|e^{m} \|^2 \leq  \sum \limits_{n=1}^{m-1}6c_1\tau_f\|e^n \|^2+6\widetilde{C}T(\tau_f^2+h_f^4)^2
     \leq  6\widetilde{C}T(\tau_f^2+h_f^4)^2\exp (6c_1{m}\tau_f)
     \leq  \widetilde{C}(\tau_f^2+h_f^4)^2.
   \end{split}
  \end{equation*}
 This completes the proof.
  \end{proof}

  \subsection{Stability}
  For the 1D problem of BE, the stability of the proposed compact difference scheme has been rigorously established, see [\citenum{wangxp}, Theorem 3.6]. In 2D case, however, the classical Sobolev embedding theorem $\|v\|_\infty \leq \widetilde{C}|v|_1$ fails, precluding an analogous stability proof for the first step of ST-TGCD (or NCD) scheme \eqref{eq3.7}-\eqref{eq3.11}. To the best of our knowledge, at present, there is no proof of the stability of the compact difference scheme for PDEs with 2D Burgers’ type nonlinearity. In reference \cite{ZQF}, the authors only gave the proof of convergence for a linear compact scheme solving the 2D Sobolev equation with Burgers’ type nonlinearity, but not of stability. Although we can't prove the stability of the first step of ST-TGCD scheme here, we are able to establish its boundedness in $L^2$-norm.
\subsubsection{Boundedness of the step 1.}
\begin{theorem}
   Suppose $\{(u_c)_{pq}^r,(v_c)_{pq}^r,(v_c)_{pq}^r~|~(x_p,y_q, t_r)\in \Omega_h^c\times \Omega_\tau^c\}$ is the solution of the coarse grid system \eqref{eq3.12}-\eqref{eq3.16}, we have
   \begin{equation} \label{eq4.28}
      \|(u_c)^{m}\|_c \leq \|(u_c)^{0}\|_c, \quad 1 \leq m \leq N^c.
   \end{equation}
\end{theorem}
\begin{proof}
   Taking an inner product of \eqref{eq3.12} with $(u_c)^{r-\frac{1}{2}}~(1\leq r \leq m)$, and using Lemma \ref{lemma2.4}, we have
   \begin{equation} \label{eq4.29}
      \left\langle \delta_t^c (u_c)^{r-\frac{1}{2}}~,~(u_c)^{r-\frac{1}{2}} \right\rangle _c = \lambda \left\langle (v_c)^{r-\frac{1}{2}}+(w_c)^{r-\frac{1}{2}}~,~(u_c)^{r-\frac{1}{2}} \right\rangle _c.
   \end{equation}
   By taking $S=T=0$ in Lemma \ref{lemma2.6}, we have
   \begin{equation} \label{eq4.30}
      \left\langle (v_c)^{r-\frac{1}{2}}+(w_c)^{r-\frac{1}{2}}~,~(u_c)^{r-\frac{1}{2}} \right\rangle _c \leq  -|(u_c)^{r-\frac{1}{2}}|_{1,c}-\frac{h_c^2}{18}\Big(\|(v_c)^{r-\frac{1}{2}} \|_c^2+\|(w_c)^{r-\frac{1}{2}} \|_c^2\Big) \leq 0.
   \end{equation}
   Substituting \eqref{eq4.30} into \eqref{eq4.29}, we obtain
   \begin{equation*}
      \frac{1}{2\tau_c}\left( \|(u_c)^{r} \|_c^2-\|(u_c)^{(r-1)} \|_c^2 \right) \leq 0.
   \end{equation*}
   As long as summing up the above inequality for $r$ from $1$ to $m$, then we can obtain \eqref{eq4.28}.
\end{proof}

In the second step of the ST-TGCD scheme, the approximate solution \((u_f)_{ij}^n\) is constructed solely by Lagrange interpolation. Consequently, no additional stability analysis is required. In what follows, we shall focus on the stability of the last step (step 3).

\subsubsection{Stability of the step 3.}
Assume $\{\tilde{u}_{ij}^n,\tilde{v}_{ij}^n,\tilde{w}_{ij}^n~|~(x_{i}, y_{j})\in \Omega_h^f ,~ t_n\in\Omega_\tau^f\}$ is the solution of the following perturbation equation with periodic condition.
\begin{align}
\notag  &\delta_{t}^f \tilde{u}^{n-\frac{1}{2}}_{ij} + \psi_h^f\Big((u_f)^{n-\frac{1}{2}}_{ij}, \tilde{u}^{n-\frac{1}{2}}_{ij}\Big) - \frac{h_{f}^{2}}{2}\left[\psi_x^f\Big((v_f)^{n-\frac{1}{2}}_{ij}, \tilde{u}^{n-\frac{1}{2}}_{ij}\Big)+\psi_y^f\Big((w_f)^{n-\frac{1}{2}}_{ij}, \tilde{u}^{n-\frac{1}{2}}_{ij}\Big)\right]  \\ 
\label{eq4.31}  &\quad = \lambda(\tilde{v}^{n-\frac{1}{2}}_{ij} + \tilde{w}^{n-\frac{1}{2}}_{ij})+\varepsilon^{n-\frac{1}{2}}_{ij} , \quad (x_{i}, y_{j})\in \Omega_h^f ,\quad t_n\in\Omega_\tau^f\setminus\{0\},\\
\label{eq4.32} &\tilde{v}_{ij}^n = \delta_{xx}^f \tilde{u}_{ij}^n - \frac{h_{f}^2}{12} \delta_{xx}^f \tilde{v}_{ij}^n, \quad (x_{i}, y_{j})\in \Omega_h^f ,\quad t_n\in\Omega_\tau^f, \\
\label{eq4.33} &\tilde{w}_{ij}^n = \delta_{yy}^f \tilde{u}_{ij}^n - \frac{h_{f}^2}{12} \delta_{yy}^f \tilde{w}_{ij}^n, \quad (x_{i}, y_{j})\in \Omega_h^f ,\quad t_n\in\Omega_\tau^f,\\
\label{eq4.34} &\tilde{u}_{ij}^0 = {u}_0(x_i,y_j)+\eta_{ij}, \quad (x_{i}, y_{j})\in \Omega_h^f,
\end{align}
where \(\varepsilon^{n-\frac{1}{2}}_{ij}\) and \(\eta_{ij}\) are the samll perturbations. Subtracting \eqref{step3-1}-\eqref{step3-4} from \eqref{eq4.31}-\eqref{eq4.34}, respectively, to yield
\begin{align}
\notag  &\delta_{t}^f \hat{u}^{n-\frac{1}{2}}_{ij} + \psi_h^f\left((u_f)^{n-\frac{1}{2}}_{ij}, \hat{u}^{n-\frac{1}{2}}_{ij}\right) - \frac{h_{f}^{2}}{2}\left[\psi_x^f\left((v_f)^{n-\frac{1}{2}}_{ij}, \hat{u}^{n-\frac{1}{2}}_{ij}\right)+\psi_y^f\left((w_f)^{n-\frac{1}{2}}_{ij}, \hat{u}^{n-\frac{1}{2}}_{ij}\right)\right]  \\ 
\label{eq4.36}  &\quad = \lambda(\hat{v}^{n-\frac{1}{2}}_{ij} + \hat{w}^{n-\frac{1}{2}}_{ij})+\varepsilon^{n-\frac{1}{2}}_{ij} , \quad (x_{i}, y_{j})\in \Omega_h^f ,\quad t_n\in\Omega_\tau^f\setminus\{0\},\\
\label{eq4.37} &\hat{v}_{ij}^n = \delta_{xx}^f \hat{u}_{ij}^n - \frac{h_{f}^2}{12} \delta_{xx}^f \hat{v}_{ij}^n, \quad (x_{i}, y_{j})\in \Omega_h^f ,\quad t_n\in\Omega_\tau^f, \\
\label{eq4.38} &\hat{w}_{ij}^n = \delta_{yy}^f \hat{u}_{ij}^n - \frac{h_{f}^2}{12} \delta_{yy}^f \hat{w}_{ij}^n, \quad (x_{i}, y_{j})\in \Omega_h^f ,\quad t_n\in\Omega_\tau^f,\\
\label{eq4.39} &\hat{u}_{ij}^0 = \eta_{ij}, \quad (x_{i}, y_{j})\in \Omega_h^f,
\end{align}
where 
\begin{equation*}
   \hat{u}_{ij}^n = \tilde{u}_{ij}^n - u_{ij}^n, \quad \hat{v}_{ij}^n = \tilde{v}_{ij}^n - v_{ij}^n, \quad \hat{w}_{ij}^n = \tilde{w}_{ij}^n - w_{ij}^n.
\end{equation*}
\begin{theorem}
   Let \(\{\hat{u}_{ij}^n, \hat{v}_{ij}^n, \hat{w}_{ij}^n~|~(x_i,y_j)\in \Omega_h^f,~t_n\in\Omega_\tau^f\}\) be the solution of the residual system with periodic condition \eqref{eq4.36}-\eqref{eq4.39}. Then, for any \(1\leq n \leq N^f\), it holds that
   \begin{equation} \label{eq4.41}
      \|\hat{u}^n\| \leq \exp(2T) \Big(2\|\eta\|^2 + 2\tau_f\sum_{k=1}^{n}\|\varepsilon^{k-\frac{1}{2}} \|^2 \Big).
   \end{equation}
\end{theorem}
\begin{proof}
   Taking an inner product of \eqref{eq4.36} with \(\hat{u}^{n-\frac{1}{2}}\) and applying Lemma \ref{lemma2.4}, we have
   \begin{equation} \label{eq4.42}
      \left\langle \delta_t^f \hat{u}^{n-\frac{1}{2}}~,~\hat{u}^{n-\frac{1}{2}} \right\rangle = \lambda \left\langle \hat{v}^{n-\frac{1}{2}}+\hat{w}^{n-\frac{1}{2}}~,~\hat{u}^{n-\frac{1}{2}} \right\rangle + \left\langle \varepsilon^{n-\frac{1}{2}}~,~\hat{u}^{n-\frac{1}{2}} \right\rangle.
   \end{equation}
Applying Lemma \ref{lemma2.6}, Cauchy-Schwarz inequality and Young inequality to get
\begin{equation*}
   \begin{cases}
      \begin{split}
        & \left\langle \hat{v}^{n-\frac{1}{2}}+\hat{w}^{n-\frac{1}{2}}~,~\hat{u}^{n-\frac{1}{2}} \right\rangle \leq -|\hat{u}^{n-\frac{1}{2}}|_{1}^2-\frac{h_f^2}{18}\Big(\|\hat{v}^{n-\frac{1}{2}} \|^2+\|\hat{w}^{n-\frac{1}{2}} \|^2\Big) \leq 0,\\
        & \left\langle \varepsilon^{n-\frac{1}{2}}~,~\hat{u}^{n-\frac{1}{2}} \right\rangle \leq \|\varepsilon^{n-\frac{1}{2}} \|\|\hat{u}^{n-\frac{1}{2}} \| \leq \frac{1}{2}\|\varepsilon^{n-\frac{1}{2}} \|^2+\frac{1}{2}\|\hat{u}^{n-\frac{1}{2}} \|^2.
      \end{split}
   \end{cases}
\end{equation*}
Thus, we have
\begin{equation*}
   \frac{1}{2\tau_f}\left( \|\hat{u}^n \|^2-\|\hat{u}^{n-1} \|^2 \right) \leq \frac{1}{2}\|\varepsilon^{n-\frac{1}{2}} \|^2+\frac{1}{2}\|\hat{u}^{n-\frac{1}{2}} \|^2\leq \frac{1}{2}\|\varepsilon^{n-\frac{1}{2}} \|^2+\frac{1}{4}\left( \|\hat{u}^n \|^2+\|\hat{u}^{n-1} \|^2 \right).
\end{equation*}
Summing the above inequality from \(n=1\) to \(m\), we have
\begin{equation*}
   \begin{split}
      \|\hat{u}^m \|^2 -\|\eta\|^2 \leq \frac{1}{2}\tau_f\|\hat{u}^m \|^2+ \tau_f \sum_{n=0}^{m-1}\|\hat{u}^n \|^2-\frac{1}{2}\tau_f\|\eta\|^2+ \tau_f\sum_{n=1}^{m}\|\varepsilon^{n-\frac{1}{2}} \|^2
   \end{split}
\end{equation*}
As long as \(\tau_f\leq 1\), then we get
$$
\|\hat{u}^m \|^2 \leq 2\tau_f \sum_{n=0}^{m-1}\|\hat{u}^n \|^2 +2\|\eta\|^2 + 2\tau_f\sum_{n=1}^{m}\|\varepsilon^{n-\frac{1}{2}} \|^2.
$$
Applying the discrete Gr\"{o}nwall inequality, we obtain
\begin{equation*}
   \|\hat{u}^m \|^2 \leq \Big(2\|\eta\|^2 + 2\tau_f\sum_{n=1}^{m}\|\varepsilon^{n-\frac{1}{2}} \|^2 \Big) \exp(2\tau_f m)\leq \exp(2T) \Big(2\|\eta\|^2 + 2\tau_f\sum_{n=1}^{m}\|\varepsilon^{n-\frac{1}{2}} \|^2 \Big), ~~ 1\leq m \leq N^f.
\end{equation*}
   This completes the proof.
\end{proof}

\section{Numerical Experiment} \label{sec5}
This section presents numerical experiments designed to validate the effectiveness of the proposed Sapce-Time Two-Grid Compact Difference (ST-TGCD) scheme \eqref{eq3.12}-\eqref{step3-5}. To highlight its advantages, we compare its performance with the standard Nonlinear Compact Difference (NCD) scheme \eqref{eq3.7}-\eqref{eq3.11}. Before presenting the numerical results, we assume the problems considered here on a square domain, namely, $L_1=L_2$, which means $M_1^f=M_2^f=:M^f$. All computations were performed on a system running Windows 11, 64-bit with a 12th Gen Intel(R) Core(TM) i7-12700 CPU @ 2.10 GHz and 16.0 GB of RAM, utilizing MATLAB R2023b.

  Let us use $E_{1}(h_f, \tau_f)$, $R_{1}^h$ and $R_{1}^\tau $ to denote NCD scheme's $L^2$-norm error, spatial convergence order and temporal convergence order, respectively.
   \begin{equation*}
     \begin{split}
       &E_{1}(h_f, \tau_f):=\sqrt{h_f^2\sum\limits_{i=1}^{M^f}\sum\limits_{j=1}^{M^f}\left(U_{ij}^{N^f}(h_f, \tau_f)-u_{ij}^{N^f}(h_f, \tau_f)\right)^2},\\
        &R_{1}^h:=\log_2 \left( \frac{E_{1}(h_f, \tau_f)}{E_{1}(h_f/2, \tau_f)} \right),\quad R_{1}^\tau:=\log_2 \left( \frac{E_{1}(h_f, \tau_f)}{E_{1}(h_f,\tau_f/2)} \right),
     \end{split}
   \end{equation*}
   where $U$ denotes the exact solution of the problem \eqref{main_equation}-\eqref{boundary_condition}, and $u$ denotes the numerical solution from the NCD scheme. Analogously, $E_{2}(h_f, \tau_f)$, $\,R_{2}^h$ and $R_{2}^\tau$ can be defined as the ST-TGCD scheme's $L^2$-norm error, spatial convergence order and temporal convergence order, respectively.
   \vskip 0.2mm
   Both the NCD scheme and the coarse-grid stage \eqref{eq3.12}-\eqref{eq3.16} within the ST-TGCD scheme are solved using a fixed-point iteration. Taking the NCD scheme as an example, the iteration proceeds as follows (using the vector notation $\textbf{u}^n$, $\textbf{v}^n$, and $\textbf{w}^n$
  defined for the solution values at time layer $t_n$)
  \newcommand{\tbfu}{\textbf{u}}
   \begin{equation*}
     \begin{cases}
     \frac{1}{\tau_f}(\tbfu^{n,k+1}-\tbfu^{n-1})+\psi_h^f(\tbfu^{n-\frac{1}{2}},\tbfu^{n-\frac{1}{2}})-\frac{h_f^2}{2}\psi_x^f(\textbf{v}^{n-\frac{1}{2}},\tbfu^{n-\frac{1}{2}})\\
     \qquad-\frac{h_f^2}{2}\psi_y^f(\textbf{w}^{n-\frac{1}{2}},\tbfu^{n-\frac{1}{2}})=\frac{\lambda}{2}(\textbf{v}^{n,k+1}+\textbf{v}^{n-1}+\textbf{w}^{n,k+1}+\textbf{w}^{n-1}),\\
     \\
     \psi_h^f(\tbfu^{n-\frac{1}{2}},\tbfu^{n-\frac{1}{2}})=\frac{1}{4}\left( \psi_h^f(\tbfu^{n,k},\tbfu^{n,k+1})+\psi_h^f(\tbfu^{n,k+1},\tbfu^{n-1})+\psi_h^f(\tbfu^{n-1},\tbfu^{n,k+1})+\psi_h^f(\tbfu^{n-1},\tbfu^{n-1})\right),\\
     \\
     \psi_x^f(\textbf{v}^{n-\frac{1}{2}},\tbfu^{n-\frac{1}{2}})=\frac{1}{4}\left( \psi_x^f(\textbf{v}^{n,k},\tbfu^{n,k+1})+\psi_x^f(\textbf{v}^{n,k+1},\tbfu^{n-1})+\psi_x^f(\textbf{v}^{n-1},\tbfu^{n,k+1})+\psi_x^f(\textbf{v}^{n-1},\tbfu^{n-1})\right),\\
     \\
     \psi_y^f(\textbf{w}^{n-\frac{1}{2}},\tbfu^{n-\frac{1}{2}})=\frac{1}{4}\left( \psi_y^f(\textbf{w}^{n,k},\tbfu^{n,k+1})+\psi_y^f(\textbf{w}^{n,k+1},\tbfu^{n-1})+\psi_y^f(\textbf{w}^{n-1},\tbfu^{n,k+1})+\psi_y^f(\textbf{w}^{n-1},\tbfu^{n-1})\right),\\
     \\
   \textbf{v}^{n}=\left(I+\frac{h_f^2}{12}\delta_{xx}^f \right)^{-1}\delta_{xx}^f \tbfu^{n}, \quad \textbf{w}^{n}=\left(I+\frac{h_f^2}{12}\delta_{yy}^f \right)^{-1}\delta_{yy}^f \tbfu^{n},
   \end{cases}
   \end{equation*}
   for $1\leq n \leq N^f.$ Here $k$ denotes the iteration index, $I$ denotes the identity operator, and the initial guess is  $\tbfu^{n,0}=\tbfu^{n-1}$. The iteration stops when $\max\limits_{1\leq i,j\leq M^f}| u_{ij}^{n,k+1}-u_{ij}^{n,k} |\leq 10^{-8}$ or after a maximum of 100 iterations.

   \textbf{Example 1.} (\textit{Manufactured Solution}) We first consider a test case with a presumed exact solution to directly measure errors. The exact solution is
   $$  u(x,y,t)=e^{-t}\sin(\pi x)sin(\pi y),\qquad (x,y,t)\in [0,2]\times[0,2]\times[0,1]. $$
 The initial condition is $u_0(x,y)=\sin(\pi x)\sin(\pi y)$, and the corresponding source term is given as follows
   $$ f(x,y,t)=e^{-t}\sin(\pi x)\sin(\pi y)\left[-1+\pi e^{-t}\sin\Big(\pi (x+y)\Big)+\lambda \pi^2\right]. $$

We solve this problem using both the NCD and ST-TGCD schemes and compare the numerical solutions against the exact solution. The results are summarized in Tables \ref{tb1}-\ref{tb3} and Figures \ref{fig2}-\ref{fig3}.

 \begin{table}
   \center
   \caption{$L^{2}$-norm error, temporal convergence order and CPU time (seconds) behaviors versus time mesh reduction with the fixed spatial-size $h_f=1/50\;(M^f=100)$, spatial step-size ratio $k_h=2$ and temporal step-size ratio $k_\tau=2$.} \label{tb1}
   \begin{tabular}{ccccccccc}
     \toprule
    & &   & \multicolumn{3}{c}{NCD Scheme \eqref{eq3.7}-\eqref{eq3.11}} &\multicolumn{3}{c}{ST-TGCD Scheme \eqref{eq3.12}-\eqref{step3-5}} \\
     \cmidrule(r){4-6}  \cmidrule(r){7-9}
     $\lambda$& $\tau_c$ &$ \tau_f$& $E_{1}$ & $R_{1}^\tau$ & CPU(s) & $E_{2}$ & $R_{2}^\tau$ & CPU(s)\\
    \midrule
    \multirow{4}{*}{\centering 1}
     &1/4  &1/8  &3.3085e-05 & *    &181.84 & 2.5942e-05 & *      &43.58  \\
     &1/8  &1/16 &8.4168e-06 &1.9748&340.56 & 6.3830e-06 & 2.0230 &85.84   \\
     &1/16 &1/32 &2.0986e-06 &2.0039&573.22 & 1.5822e-06 & 2.0123 &161.35  \\
     &1/32 &1/64 &5.1815e-07 &2.0180&1130.50& 3.8097e-07 & 2.0542 &311.88   \\
      \midrule
     \multirow{4}{*}{\centering 0.1}
     &1/4  &1/8  & 3.7122e-04 & *      &276.00  &4.7122e-04 & *    &42.84\\
     &1/8  &1/16 & 9.2946e-05 & 1.9978 &356.44  &1.1766e-04 &2.0018&78.98 \\
     &1/16 &1/32 & 2.3333e-05 & 1.9940 &1095.10 &2.9672e-05 &1.9874&154.71 \\
     &1/32 &1/64 & 5.9296e-06 & 1.9764 &1698.90 &7.7165e-06 &1.9431&302.90  \\
     \bottomrule
   \end{tabular}
 \end{table}

 \begin{table}
   \center
   \caption{$L^{2}$-norm error, temporal convergence order and CPU time (seconds) behaviors versus time mesh reduction with the fixed spatial-size $h_f=1/50$, spatial step-size ratio $k_h=2$ and temporal step-size ratio $k_\tau=3$.} \label{tb2}
   \begin{tabular}{ccccccccc}
     \toprule
        & &   & \multicolumn{3}{c}{NCD Scheme \eqref{eq3.7}-\eqref{eq3.11}} &\multicolumn{3}{c}{ST-TGCD Scheme \eqref{eq3.12}-\eqref{step3-5}} \\
     \cmidrule(r){4-6}  \cmidrule(r){7-9}
     $\lambda$& $\tau_c$ &$ \tau_f$& $E_{1}$ & $R_{1}^\tau$ & CPU(s) & $E_{2}$ & $R_{2}^\tau$ & CPU(s)\\
    \midrule
 
    \multirow{4}{*}{\centering 1}
     &1/4  &1/12 & 1.4957e-05 & *      &264.49  &1.2566e-05 & *    &62.67\\
     &1/8  &1/24 & 3.7374e-06 & 2.0007 &672.26  &3.8342e-06 &1.7126&119.95 \\
     &1/16 &1/48 & 9.2777e-07 & 2.0102 &1226.90 &1.0779e-06 &1.8307&235.37 \\
     &1/32 &1/96 & 2.2621e-07 & 2.0361 &1655.80 &2.4829e-07 &2.1181&462.60  \\
     \midrule

     \multirow{4}{*}{\centering 0.1}
     &1/4  &1/12 & 1.6512e-04 & *      &399.76  &4.8617e-04 & *    &59.45\\
     &1/8  &1/24 & 4.1382e-05 & 1.9964 &740.08  &1.2104e-04 &2.0060&116.02 \\
     &1/16 &1/48 & 1.0441e-05 & 1.9868 &943.10  &3.0414e-05 &1.9927&220.31 \\
     &1/32 &1/96 & 2.7090e-06 & 1.9464 &1745.60 &7.8167e-06 &1.9601&432.46  \\
     \bottomrule
   \end{tabular}
 \end{table}

 \begin{table}
   \center
   \caption{$L^{2}$-norm error, spatial convergence order and CPU time (seconds) behaviors versus space-mesh reduction with the fixed spatial-size $\tau_f=1/512$, temporal step-size ratio $k_\tau=2$ and spatial step-size ratio $k_h=2$.} \label{tb3}
   \begin{tabular}{ccccccccc}
     \toprule
        & &   & \multicolumn{3}{c}{NCD Scheme \eqref{eq3.7}-\eqref{eq3.11}} &\multicolumn{3}{c}{ST-TGCD Scheme \eqref{eq3.12}-\eqref{step3-5}} \\
     \cmidrule(r){4-6}  \cmidrule(r){7-9}
     $\lambda$& $h_c$ &$ h_f$& $E_{1}$ & $R_{1}^h$ & CPU(s) & $E_{2}$ & $R_{2}^h$ & CPU(s)\\
    \midrule

    \multirow{4}{*}{\centering 1}
     &1/4  &1/8  & 7.0303e-04 & *      &0.35   &7.9502e-04 &*      &0.19  \\
     &1/8  &1/16 & 4.4032e-05 & 3.9970 &2.93   &6.9166e-05 &3.5229 &1.15\\
     &1/16 &1/32 & 2.7553e-06 & 3.9983 &30.58  &4.5795e-06 &3.9168 &11.77 \\
     &1/32 &1/64 & 1.6991e-07 & 4.0194 &639.12 &2.8836e-07 &3.9892 &228.42 \\
      \midrule
     \multirow{4}{*}{\centering 0.2}
     &1/4  &1/8  & 2.1684e-03 & *      &0.46   &3.3272e-03  &*      &0.12  \\
     &1/8  &1/16 & 1.4371e-04 & 3.9154 &3.02   &4.1089e-04  &3.0175 &1.23 \\
     &1/16 &1/32 & 9.1756e-06 & 3.9692 &30.50  &2.8147e-05  &3.8677 &13.16 \\
     &1/32 &1/64 & 5.9249e-07 & 3.9530 &709.62 &1.8130e-06  &3.9565 &235.91 \\
     \bottomrule
   \end{tabular}
 \end{table}

  \begin{figure}
     \centering
  \includegraphics[width=0.75\textwidth]{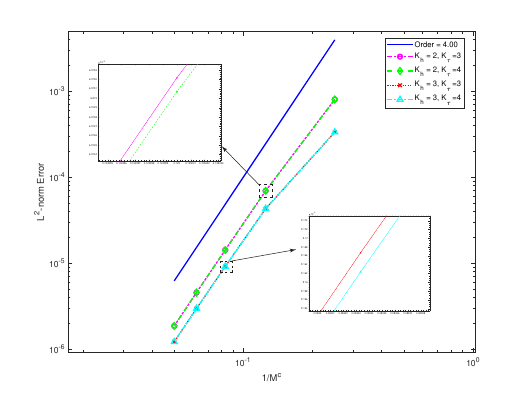}
  \caption{The spatial convergence orders of ST-TGCD scheme for different spatiotemporal step-size ratios $k_h$ and $k_\tau$, but fixed $N^c=128$, $\lambda=1$. }\label{fig2}
 \end{figure}

 \begin{figure}
     \centering
  \includegraphics[width=0.75\textwidth]{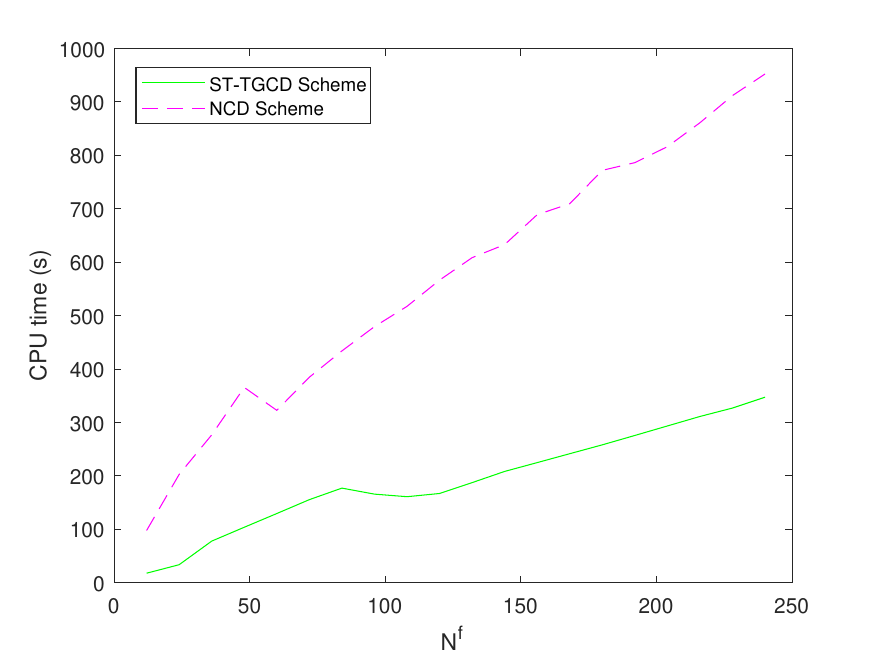}
  \caption{Comparison of CPU time cost between ST-TGCD scheme and NCD scheme, fixed $M^f=80$, $\lambda=1$, and space-time step-size ratios $k_h=2$ and $k_\tau=3$. }\label{fig3}
 \end{figure}

\vskip 1mm
\textit{Tables 1 and 2 (Temporal Refinement):} These tables present the $L^2$-norm errors, temporal convergence orders $(R_{1}^\tau, R_{2}^\tau)$, and CPU time for different viscosity coefficients $\lambda$ and temporal step-size ratios $k_\tau=2$ and $k_\tau=3$, respectively. The spatial mesh size is fixed at $h_f=1/50\;(M^f=100)$, and the spatial step-size ratio is $k_h=2$. The key observations are:
\begin{itemize}
   \item The errors computed by both schemes are very close, indicating comparable accuracy.
   \item Both schemes achieve approximately second-order temporal convergence.
   \item The computational time (CPU(s)) of the ST-TGCD scheme is consistently and significantly lower than that of the NCD scheme. This time saving becomes more pronounced as the mesh is refined. The data show that CPU time is reduced by more than 70\%.
\end{itemize}

\vskip 1mm
\textit{Table 3 (Spatial Refinement):} This table shows the $L^2$-norm errors, spatial convergence orders $(R_{1}^h, R_{2}^h)$, and CPU time for spatial mesh refinement with temporal mesh size fixed at $\tau_f=1/512$, the temporal step-size ratio $k_\tau=2$, and step-size ratio $k_h=2$. The results demonstrate:
\begin{itemize}
   \item Comparable errors between the two schemes.
   \item Both schemes achieve fourth-order spatial convergence across various $\lambda$ values.
   \item The ST-TGCD scheme consistently exhibits markedly lower CPU time than the NCD scheme.
\end{itemize}

 Figure \ref{fig2} shows that the ST-TGCD scheme attains fourth-order spatial convergence for various space-time step-size ratios, exhibiting a rapid convergence rate. Figure \ref{fig3} further reveals that the ST-TGCD algorithm markedly reduces computational time, thereby substantially enhancing overall efficiency.

   \vskip 1mm
   \textbf{Example 2} (\textit{Unknown Exact Solution}). In this example, we set the initial condition as $u_0(x)=\sin(\pi x)\cos(\pi y)$ within $\Omega=[0,1]\times[1/2,3/2]$, and the final time $T=1$. As the exact solution of the system (1.1)-(1.3) is unknown, we estimate convergence rates by comparing solutions on successively refined meshes. Let $ \{u_{ij}^n\} $ be the solution of the NCD scheme using spatial division $M^f$ and temporal division $N^f$. Let $ \{\hat{u}_{ij}^n\} $ denote the solution computed with $2M^f$ and $N^f$, and $ \{\tilde{u}_{ij}^n\} $ the solution computed with $M^f$ and $2N^f$. We define the estimated errors and convergence orders as:
  $$ E_{3}^\tau(h_f,\tau_f):=\sqrt{h_f^2\sum\limits_{i=1}^{M^f}\sum\limits_{j=1}^{M^f}\left( u_{ij}^{N^f}(h_f,\tau_f) -\tilde{u}_{ij}^{2N^f}(h_f,\tau_f/2)\right)^2},\quad R_{3}^{\tau}:=\log_2\left(\frac{E_{3}^\tau(h_f,\tau_f)}{E_{3}^\tau(h_f,\tau_f/2)}\right),$$
  $$ E_{3}^h(h_f,\tau_f):=\sqrt{h_f^2\sum\limits_{i=1}^{M^f}\sum\limits_{j=1}^{M^f}\left( u_{ij}^{N^f}(h_f,\tau_f) -\hat{u}_{2i,2j}^{N^f}(h_f/2,\tau_f)\right)^2},\quad R_{3}^{h}:=\log_2\left(\frac{E_{3}^h(h_f,\tau_f)}{E_{3}^h(h_f/2,\tau_f)}\right).$$
  The corresponding ST-TGCD scheme's metrics $E_{4}^\tau(h_f,\tau_f)$, $E_{4}^h(h_f,\tau_f)$, $R_{4}^\tau$ and $R_{4}^h$ are defined similarly. The numerical results of this example are listed in Tables \ref{tb4}-\ref{tb6} and Figure \ref{fig4and5}.
  
\textit{Accuracy and Computational Efficiency:} Tables \ref{tb4}-\ref{tb6} (showing results analogous to Tables \ref{tb1}-\ref{tb3} but using the estimated errors defined above) confirm that the ST-TGCD scheme can arrive at comparable errors compared with the NCD scheme, while the former scheme achieves these results with substantially lower computational cost compared to the latter scheme, highlighting the efficiency gain of the two-grid approach. {(b)} of Figure \ref{fig4and5} can also intuitively reflect this advantage.

\textit{Convergence Estimates:} Tables \ref{tb4}-\ref{tb6} and {(a)} of Figure \ref{fig4and5} confirm that both schemes achieve the expected second-order temporal and fourth-order spatial convergence rates, which further verify Theorems \ref{Th4.2} and \ref{Th4.4}.

 \begin{table}
   \center
   \caption{$L^{2}$-norm error, temporal convergence order and CPU time (seconds) behaviors versus time mesh reduction with the fixed spatial-size $h_f=1/50$, spatial step-size ratio $k_h=2$ and temporal step-size ratio $k_\tau=2$.} \label{tb4}
   \begin{tabular}{ccccccccc}
     \toprule
        & &   & \multicolumn{3}{c}{NCD Scheme \eqref{eq3.7}-\eqref{eq3.11}} &\multicolumn{3}{c}{ST-TGCD Scheme \eqref{eq3.12}-\eqref{step3-5}} \\
     \cmidrule(r){4-6}  \cmidrule(r){7-9}
     $\lambda$& $\tau_c$ &$ \tau_f$& $E_{3}^\tau$ & $R_{3}^\tau$ & CPU(s) & $E_{4}^\tau$ & $R_{4}^\tau$ & CPU(s)\\
    \midrule

    \multirow{4}{*}{\centering 1}
     &1/64   &1/128  &5.2944e-11 & *      &57.70   &5.3336e-11 & *    &19.62\\
     &1/128  &1/256  &9.7471e-12 & 2.4414 &121.86  &9.7488e-12 &2.4518&38.41 \\
     &1/256  &1/512  &2.4459e-12 & 1.9946 &143.90  &2.4464e-12 &1.9946&76.99 \\
     &1/512  &1/1024 &6.1206e-13 & 1.9986 &277.48  &6.1216e-13 &1.9986&152.97  \\
     \midrule

     \multirow{4}{*}{\centering 0.1}
     &1/64   &1/128  & 2.3613e-06 & *      &63.40  &3.1914e-06 & *    &21.24\\
     &1/128  &1/256  & 5.9038e-07 & 1.9999 &113.25 &7.9788e-07 &1.9999&41.65 \\
     &1/256  &1/512  & 1.4768e-07 & 1.9992 &227.18 &1.9947e-07 &2.0000&82.70 \\
     &1/512  &1/1024 & 3.6955e-08 & 1.9986 &435.07 &4.9866e-08 &2.0001&163.21  \\
     \bottomrule
   \end{tabular}
 \end{table}

  \begin{table}
   \center
   \caption{$L^{2}$-norm error, temporal convergence order and CPU time (seconds) behaviors versus time mesh reduction with the fixed spatial-size $h_f=1/60$, spatial step-size ratio $k_h=3$ and temporal step-size ratio $k_\tau=4$.} \label{tb5}
   \begin{tabular}{ccccccccc}
     \toprule
        & &   & \multicolumn{3}{c}{NCD Scheme \eqref{eq3.7}-\eqref{eq3.11}} &\multicolumn{3}{c}{ST-TGCD Scheme \eqref{eq3.12}-\eqref{step3-5}} \\
     \cmidrule(r){4-6}  \cmidrule(r){7-9}
     $\lambda$& $\tau_c$ &$ \tau_f$& $E_{3}^\tau$ & $R_{3}^\tau$ & CPU(s) & $E_{4}^\tau$ & $R_{4}^\tau$ & CPU(s)\\
    \midrule
  
    \multirow{4}{*}{\centering 1}
     &1/32   &1/128  & 9.5929e-11 & *      &85.39   &1.8058e-10 & *    &44.14\\
     &1/64   &1/256  & 9.7471e-12 & 3.2989 &162.45  &9.7556e-12 &4.2102&87.62 \\
     &1/128  &1/512  & 2.4459e-12 & 1.9946 &345.23  &2.4481e-12 &1.9946&171.38 \\
     &1/256  &1/1024 & 6.1206e-13 & 1.9986 &666.94  &6.1259e-13 &1.9987&342.89  \\
     \midrule
   
     \multirow{4}{*}{\centering 0.1}
     &1/32   &1/128  & 2.3613e-06 & *      &135.31  &6.5266e-06 & *    &43.57\\
     &1/64   &1/256  & 5.9038e-07 & 1.9999 &245.17  &1.6323e-06 &1.9994&86.30 \\
     &1/128  &1/512  & 1.4768e-07 & 1.9992 &628.02  &4.0811e-07 &1.9999&172.76 \\
     &1/256  &1/1024 & 3.6955e-08 & 1.9986 &1114.80 &1.0203e-07 &1.9999&345.35  \\
     \bottomrule
   \end{tabular}
 \end{table}

 \begin{table}
   \center
   \caption{$L^{2}$-norm error, spatial convergence order and CPU time (seconds) behaviors versus space-mesh reduction with the fixed spatial-size $\tau_f=1/512$, temporal step-size ratio $k_\tau=2$ and spatial step-size ratio $k_h=2$.} \label{tb6}
   \begin{tabular}{ccccccccc}
     \toprule
        & &   & \multicolumn{3}{c}{NCD Scheme \eqref{eq3.7}-\eqref{eq3.11}} &\multicolumn{3}{c}{ST-TGCD Scheme \eqref{eq3.12}-\eqref{step3-5}} \\
     \cmidrule(r){4-6}  \cmidrule(r){7-9}
     $\lambda$& $h_c$ &$ h_f$& $E_{3}^h$ & $R_{3}^h$ & CPU(s) & $E_{4}^h$ & $R_{4}^h$ & CPU(s)\\
    \midrule
    
    \multirow{4}{*}{\centering 1}
     &1/4  &1/8  & 2.4736e-12 & *      &0.29   &2.4900e-12 & *    &0.17\\
     &1/8  &1/16 & 1.5376e-13 & 4.0079 &2.00   &1.5581e-13 &3.9983&1.14 \\
     &1/16 &1/32 & 9.5982e-15 & 4.0017 &19.89  &9.7391e-15 &3.9999&12.02 \\
     &1/32 &1/64 & 7.3914e-16 & 3.6988 &509.09 &7.4745e-16 &3.7037&226.68 \\
     \midrule
     \multirow{4}{*}{\centering 0.2}
     &1/4  &1/8  & 5.2671e-06 & *      &0.32  &6.2398e-06 & *    &0.37\\
     &1/8  &1/16 & 3.3347e-07 & 3.9814 &2.88  &5.2386e-07 &3.5743&1.17 \\
     &1/16 &1/32 & 2.0913e-08 & 3.9951 &28.34 &3.4382e-08 &3.9294&11.89 \\
     &1/32 &1/64 & 1.3080e-09 & 3.9990 &609.08&2.1582e-09 &3.9937&233.52 \\
     \bottomrule
   \end{tabular}
 \end{table}

\begin{figure}
   \centering
\subfloat[ ]{
   \label{fig:subfig:a}
   \includegraphics[width=0.47\textwidth]{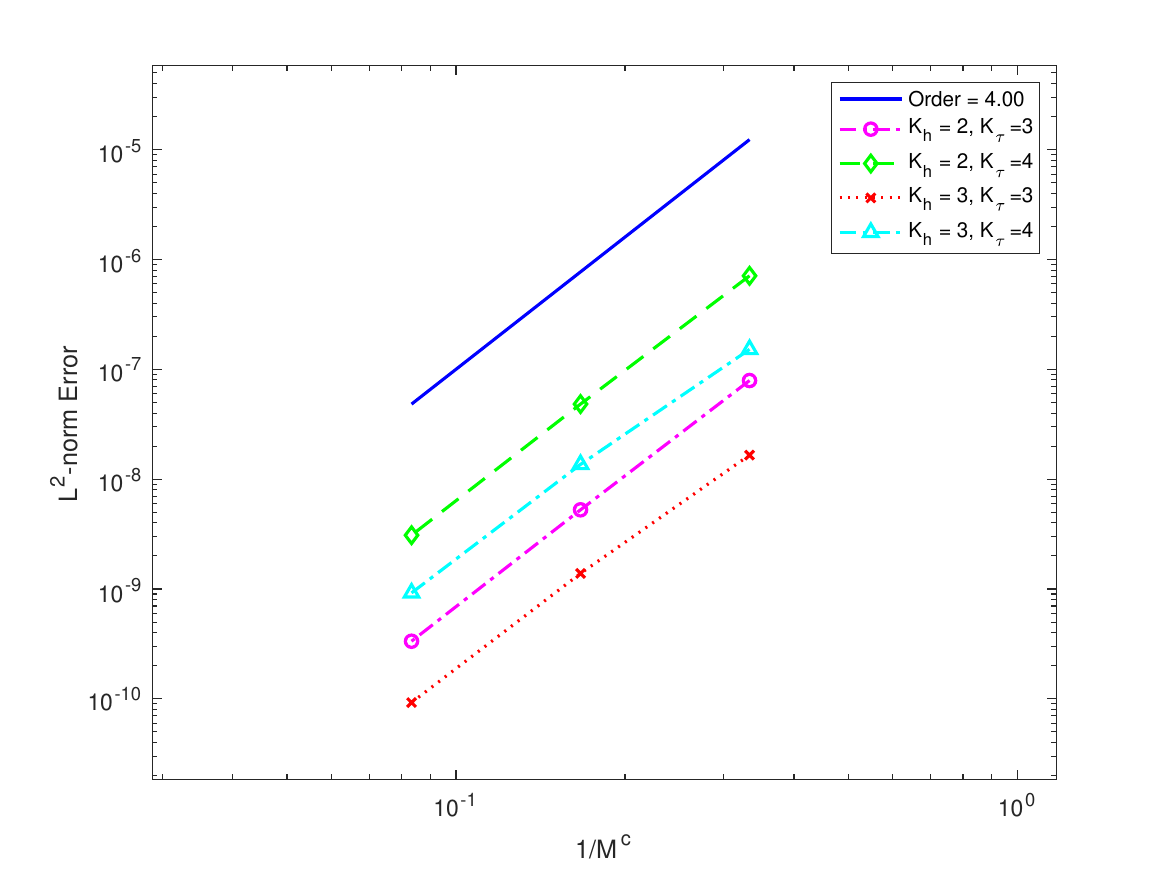}
}
\hspace{0.2em}
\subfloat[ ]{
   \label{fig:subfig:b}
   \includegraphics[width=0.47\textwidth]{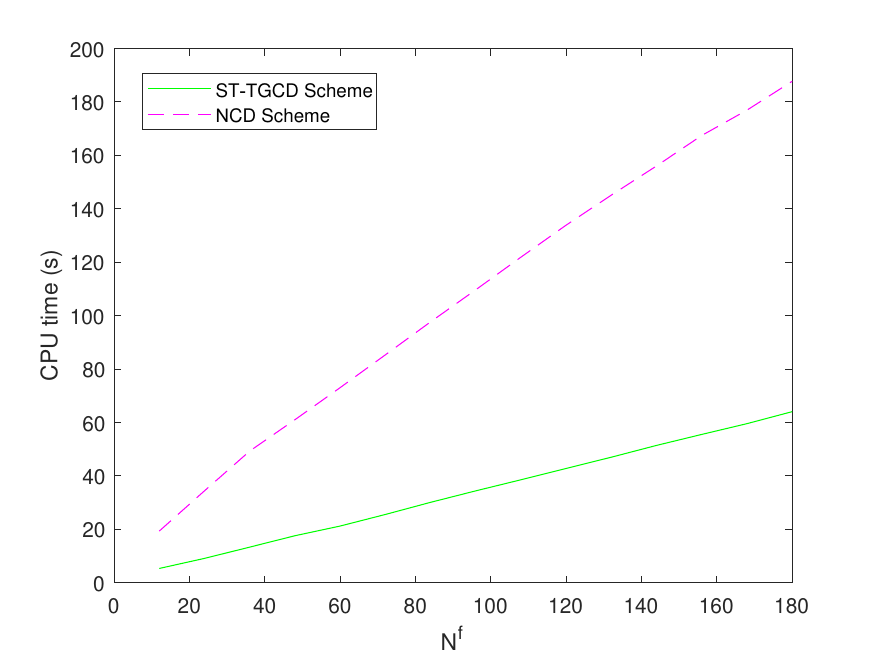}
}
\caption{{(a)} is the spatial convergence orders of ST-TGCD scheme for different space-time step-size ratios $k_h$ and $k_\tau$, but fixed $N^c=128$, $\lambda=0.5$. {(b)} is the comparison of CPU time cost between ST-TGCD scheme and NCD scheme, fixed $M^f=60$, $\lambda=0.1$, and spatiotemporal step-size ratios $k_h=2$ and $k_\tau=3$. }\label{fig4and5}
\end{figure}

\section{Concluding} \label{sec6}
We have constructed and analyzed the ST-TGCD method for the 2D viscous Burgers' equation. By combining a coarse-grid nonlinear solver with a fine-grid linearized compact correction, the algorithm attains second-order accuracy in time and fourth-order accuracy in space while reducing CPU time by more than 70\% relative to the standard nonlinear compact scheme. Theoretical proofs of unconditional convergence have been provided, and numerical experiments confirm the predicted convergence rates and the significant gain in efficiency. These results demonstrate that the ST-TGCD approach is a reliable and high efficient for solving the 2D viscous Burgers’ equation. Based on the framework of ST-TGCD presented in this paper, similar space-time compact difference schemes can be constructed for solving the KdV equation, BBMB equation, KS equation. Furthermore, the corresponding theoretical analysis can be derived in a similar manner.

 \section*{Acknowledgements} 
The authors are grateful to Professor Dongling Wang of Xiangtan University for his helpful discussions and valuable suggestions on this work.

   \bibliographystyle{plain}
   \bibliography{main}

\begin{thebibliography}{10}

\bibitem{akrivis1992finite}
G.~D Akrivis.
\newblock Finite difference discretization of the {K}uramoto-{S}ivashinsky
  equation.
\newblock {\em Numerische Mathematik}, 63(1):1--11, 1992.

\bibitem{ALI1992325}
A.H.A. Ali, G.A. Gardner, and L.R.T. Gardner.
\newblock A collocation solution for {B}uregrs' equation using cubic {B}-spline
  finite elements.
\newblock {\em Computer Methods in Applied Mechanics and Engineering},
  100(3):325--337, 1992.

\bibitem{arora2013numerical}
G.~Arora and B.~K. Singh.
\newblock Numerical solution of {B}uregrs’ equation with modified cubic
  {B}-spline differential quadrature method.
\newblock {\em Applied Mathematics and Computation}, 224:166--177, 2013.

\bibitem{bahadir2003fully}
A~R. Bahad{\i}r.
\newblock A fully implicit finite-difference scheme for two-dimensional
  {B}uregrs’ equations.
\newblock {\em Applied Mathematics and Computation}, 137(1):131--137, 2003.

\bibitem{bateman1915some}
H.~Bateman.
\newblock Some recent researches on the motion of fluids.
\newblock {\em Monthly Weather Review}, 43(4):163--170, 1915.

\bibitem{Brio}
M.~Brio and J.~Hunter.
\newblock Mach reflection for the two-dimensional {B}uregrs equation.
\newblock {\em Physica D: Nonlinear Phenomena}, 60(1-4):194--207, 1992.

\bibitem{burgers1948mathematical}
J.~M. Burgers.
\newblock A mathematical model illustrating the theory of turbulence.
\newblock {\em Advances in Applied Aechanics}, 1:171--199, 1948.

\bibitem{caldwell1987solution}
J.~Caldwell, P.~Wanless, and AE~Cook.
\newblock Solution of {B}uregrs' equation for large {R}eynolds number using
  finite elements with moving nodes.
\newblock {\em Applied Mathematical Modelling}, 11(3):211--214, 1987.

\bibitem{chen2023two}
H.~Chen, O.~Nikan, W.~Qiu, and Z.~Avazzadeh.
\newblock Two-grid finite difference method for 1{D} fourth-order
  {S}obolev-type equation with {B}uregrs’ type nonlinearity.
\newblock {\em Mathematics and Computers in Simulation}, 209:248--266, 2023.

\bibitem{daug2005numerical}
{\.I}.~Da{\u{g}}, D.~Irk, and B.~Saka.
\newblock A numerical solution of the {B}uregrs' equation using cubic
  {B}-splines.
\newblock {\em Applied Mathematics and Computation}, 163(1):199--211, 2005.

\bibitem{Fletcher}
C.~A. Fletcher.
\newblock Generating exact solutions of the two-dimensional {B}uregrs'
  equations.
\newblock {\em International Journal for Numerical Methods in Fluids},
  3:213--216, 1983.

\bibitem{gao2025efficient}
J.~Gao, S.~He, B.~Eerdun, C.~He, and J.~He.
\newblock An efficient space-time two-grid difference approach to symmetric
  regularized long waves: {E}nhanced efficiency and accuracy.
\newblock {\em Alexandria Engineering Journal}, 121:53--65, 2025.

\bibitem{gray2006toeplitz}
R.~M Gray.
\newblock Toeplitz and circulant matrices: A review.
\newblock {\em Foundations and Trends{\textregistered} in Communications and
  Information Theory}, 2(3):155--239, 2006.

\bibitem{Guo1}
B.~Guo.
\newblock Error estimation of {H}ermite spectral method for nonlinear partial
  differential equations.
\newblock {\em Mathematics of Computation}, 68(227):1067--1078, 1999.

\bibitem{Guo2}
B.~Guo, H.~Ma, and E.~Tadmor.
\newblock Spectral vanishing viscosity method for nonlinear conservation laws.
\newblock {\em SIAM Journal on Numerical Analysis}, 39(4):1254--1268, 2001.

\bibitem{holden1999operator}
H.~Holden, K.~H. Karlsen, and N.~H. Risebro.
\newblock Operator splitting methods for generalized {K}orteweg--de {V}ries
  equations.
\newblock {\em Journal of Computational Physics}, 153(1):203--222, 1999.

\bibitem{holden2013operator}
H.~Holden, C.~Lubich, and N.~Risebro.
\newblock Operator splitting for partial differential equations with {B}uregrs
  nonlinearity.
\newblock {\em Mathematics of Computation}, 82(281):173--185, 2013.

\bibitem{HuX}
X.~Hu, P.~Huang, and X.~Feng.
\newblock Two-grid method for {B}uregrs' equation by a new mixed finite element
  scheme.
\newblock {\em Mathematical Modelling and Analysis}, 19(1):1--17, 2014.

\bibitem{kreiss1986convergence}
G.~Kreiss and H.-O. Kreiss.
\newblock Convergence to steady state of solutions of {B}urgers' equation.
\newblock {\em Applied Numerical Mathematics}, 2(3-5):161--179, 1986.

\bibitem{Kundu}
S.~Kundu and A.~K. Pani.
\newblock Finite element approximation to global stabilization of the
  {B}uregrs' equation by {N}eumann boundary feedback control law.
\newblock {\em Advances in Computational Mathematics}, 44(2):541--570, 2018.

\bibitem{kutluay2004numerical}
S.~Kutluay, A.~Esen, and I.~Dag.
\newblock Numerical solutions of the {B}uregrs’ equation by the least-squares
  quadratic {B}-spline finite element method.
\newblock {\em Journal of Computational and Applied Mathematics},
  167(1):21--33, 2004.

\bibitem{Laforgue}
J.~G.~L. Laforgue and R.~E. O'Malley~Jr.
\newblock Shock layer movement for {B}urgers' equation.
\newblock {\em SIAM Journal on Applied Mathematics}, 55(2):332--347, 1995.

\bibitem{LBYtwo}
B.~Li.
\newblock A bounded numerical solution with a small mesh size implies existence
  of a smooth solution to the {N}avier--{S}tokes equations.
\newblock {\em Numerische Mathematik}, 147(2):283--304, 2021.

\bibitem{miles1981korteweg}
J.~W Miles.
\newblock The {K}orteweg-de {V}ries equation: a historical essay.
\newblock {\em Journal of Fluid Mechanics}, 106:131--147, 1981.

\bibitem{musha1978traffic}
T.~Musha and H.~Higuchi.
\newblock Traffic current fluctuation and the {B}uregrs equation.
\newblock {\em Japanese Journal of Applied Physics}, 17(5):811, 1978.

\bibitem{peng2024novel}
X.~Peng, W.~Qiu, J.~Wang, and L.~Ma.
\newblock A novel temporal two-grid compact finite difference scheme for the
  viscous {B}uregrs' equation.
\newblock {\em Advances in Applied Mathematics and Mechanics},
  16(6):1358--1380, 2024.

\bibitem{shi2024construction}
Y.~Shi, X.~Yang, and Z.~Zhang.
\newblock Construction of a new time-space two-grid method and its solution for
  the generalized {B}uregrs’ equation.
\newblock {\em Applied Mathematics Letters}, 158:109244, 2024.

\bibitem{Sloan}
I.~Sloan and V.~Thom{\'e}e.
\newblock Time discretization of an integro-differential equation of parabolic
  type.
\newblock {\em SIAM Journal on Numerical Analysis}, 23(5):1052--1061, 1986.

\bibitem{sun2015two}
H.~Sun and Z.~Sun.
\newblock On two linearized difference schemes for {B}uregrs’ equation.
\newblock {\em International Journal of Computer Mathematics},
  92(6):1160--1179, 2015.

\bibitem{Sunbook}
Z.~Sun.
\newblock {\em Numerical methods of partial differential equations (3rd
  Edition)}.
\newblock Science press, Beijing, 2022.

\bibitem{Varoglu}
E.~Varoḡlu and W.~Liam~F.
\newblock Space-time finite elements incorporating characteristics for the
  burgers' equation.
\newblock {\em International Journal for Numerical Methods in Engineering},
  16(1):171--184, 1980.

\bibitem{wangxp}
X.~Wang, Q.~Zhang, and Z.~Sun.
\newblock The pointwise error estimates of two energy-preserving fourth-order
  compact schemes for viscous {B}uregrs' equation.
\newblock {\em Advances in Computational Mathematics}, 47(2):1--42, 2021.

\bibitem{Xu1}
J.~Xu.
\newblock A novel two-grid method for semilinear elliptic equations.
\newblock {\em SIAM Journal on Scientific Computing}, 15(1):231--237, 1994.

\bibitem{Xu2}
J.~Xu.
\newblock Two-grid discretization techniques for linear and nonlinear {PDE}s.
\newblock {\em SIAM Journal on Numerical Analysis}, 33(5):1759--1777, 1996.

\bibitem{xu2009second}
P.~Xu and Z.~Sun.
\newblock A second-order accurate difference scheme for the two-dimensional
  {B}uregrs' system.
\newblock {\em Numerical Methods for Partial Differential Equations: An
  International Journal}, 25(1):172--194, 2009.

\bibitem{yu2002analysis}
X.~Yu.
\newblock Analysis of the stability and density waves for trafficflow.
\newblock {\em Chinese Physics}, 11(11):1128, 2002.

\bibitem{ZQF}
Q.~Zhang, Y.~Qin, and Z.~Sun.
\newblock Linearly compact scheme for 2{D} {S}obolev equation with {B}uregrs’
  type nonlinearity.
\newblock {\em Numerical Algorithms}, 91(3):1081--1114, 2022.

\end{thebibliography}

\end{document}